\documentclass[11pt]{article}
\usepackage{definitions}

\title{Restarted Halpern PDHG for Linear Programming}
\author{Haihao Lu\thanks{MIT, Sloan School of Management (haihao@mit.edu).} \and Jinwen Yang\thanks{University of Chicago, Department of Statistics (jinweny@uchicago.edu).}}
\date{September 2024}

\begin{document}

\maketitle

\begin{abstract}
    In this paper, we propose and analyze a new matrix-free primal-dual algorithm, called restarted Halpern primal-dual hybrid gradient (rHPDHG), for solving linear programming (LP). We show that rHPDHG can achieve optimal accelerated linear convergence on feasible and bounded LP. Furthermore, we present a refined analysis that demonstrates an accelerated two-stage convergence of rHPDHG over the vanilla PDHG with an improved complexity for identification and an accelerated eventual linear convergence that does not depend on the conservative global Hoffman constant. Regarding infeasible LP, we show that rHPDHG can recover infeasibility certificates with an accelerated linear rate, improving the previous convergence rates. {Furthermore, we discuss an extension of rHPDHG by adding reflection operation (which is dubbed as $\mathrm{r^2HPDHG}$), and demonstrate that it shares all theoretical guarantees of rHPDHG with an additional factor of 2 speedup in the complexity bound.} Lastly, we build up a GPU-based LP solver using rHPDHG/$\mathrm{r^2HPDHG}$, and the experiments on 383 MIPLIB instances showcase an improved numerical performance compared cuPDLP.jl.
\end{abstract}

\section{Introduction}
Linear programming (LP) is a fundamental class of optimization problems with many applications in operations research and computer science~\cite{dantzig2002linear,luenberger1984linear,boyd2004convex}. It has been extensively studied for the past 70 years. The state-of-the-art LP solvers essentially utilize either simplex method or interior-point method (IPM). They are reliable and robust enough to provide high-quality solutions to LP. The success of both algorithms relies heavily on the efficient factorization methods to solve the arising linear systems. Nevertheless, there are two major drawbacks preventing further scaling up these factorization-based algorithms: (i) Due to its sequential nature, factorization is widely believed to be challenging to leverage the massive parallelization offered by modern computing architectures such as graphic processing units (GPUs) and other distributed machines. (ii) It can be highly memory-demanding to store the factorization, which is often denser than the original constraint matrix itself. Thus, it may require remarkably more memory to store the factorization than to store the original instance.

Recent large-scale applications of LP, whose scale can go far beyond the capability of the classic simplex method and interior-point method, have sparked interest in developing new algorithms. A key feature of these algorithms is that they are matrix-free, i.e., there is no need to solve linear systems, and thus the computation bottleneck is matrix-vector multiplication. This guarantees a low per-iteration computational cost and can better take advantage of massively parallel computation. Along this line of research, several matrix-free LP solvers with first-order methods (FOMs) have been developed, for example, \href{https://github.com/google/or-tools/tree/stable/ortools/pdlp}{PDLP}/cuPDLP (\href{https://github.com/jinwen-yang/cuPDLP.jl}{cuPDLP.jl}/\href{https://github.com/COPT-Public/cuPDLP-C}{cuPDLP-C})~\cite{applegate2021practical,lu2023cupdlp,lu2023cupdlpc} based on primal-dual hybrid gradient (PDHG), matrix-free IPM solver \href{https://github.com/leavesgrp/ABIP}{ABIP}~\cite{lin2021admm, deng2022new}, dual-based solver ECLIPSE~\cite{basu2020eclipse}, general conic solver \href{https://github.com/cvxgrp/scs}{SCS}~\cite{o2016conic,o2021operator} based on alternating direction method of multiplier (ADMM), etc. Particularly, cuPDLP (implemented on GPU) was shown to have comparable performance to the state-of-the-art commercial LP solvers, such as Gurobi and COPT, on standard benchmark sets~\cite{lu2023cupdlp,lu2023cupdlpc}, which demonstrates the efficiency of this types of algorithms.

Besides the promising empirical performance, the algorithm behind cuPDLP, restarted average primal-dual hybrid gradient (raPDHG), was shown in theory to achieve the optimal linear convergence rate for solving feasible and bounded LP among a large class of first-order methods with complexity $O\pran{\frac{\gamma}{\alpha}\log\pran{\frac{1}{\varepsilon}}}$~\cite{applegate2023faster}, where $\varepsilon$ is the desired accuracy of the solution, $\gamma$ is the smoothness constant and $\alpha$ is a sharpness constant of the LP. In contrast, the vanilla primal-dual hybrid gradient algorithm was shown to have a sub-optimal linear rate with complexity $O\pran{(\frac{\gamma}{\alpha})^2\log\pran{\frac{1}{\varepsilon}}}$~\cite{lu2022infimal}. This demonstrates the importance of restarts that improve the algorithm's complexity theory, which was also observed to have a significant practical impact~\cite{applegate2021practical}. Following the terminology of Nesterov~\cite{nesterov2013introductory}, we say raPDHG achieves an accelerated linear convergence rate compared to the vanilla PDHG.

However, the sharpness constant $\alpha$ in the complexity theory of raPDHG depends on the global Hoffman constant of the KKT system corresponding to the LP~\cite{applegate2023faster}, whereas the Hoffman constant is well-known to be overly conservative and uninformative~\cite{pena2021new}. Thus, the rate derived in~\cite{applegate2023faster} can be too loose to interpret the successful empirical behavior of the algorithm. Motivated, a refined trajectory-based analysis of vanilla PDHG was developed in \cite{lu2023geometry}. Compared with the previous work~\cite{lu2022infimal,applegate2023faster}, the complexity results derived in \cite{lu2023geometry} for vanilla PDHG do not rely on any global Hoffman constant. Moreover, it characterizes a two-stage convergence behavior and the geometry of the algorithm: in stage I, PDHG identifies active variables and the length of the first stage is driven by a certain quantity, which measures how close the non-degeneracy part of the LP instance is to degeneracy, while in stage II, the algorithm effectively solves a homogeneous linear inequality system, and the complexity of the second stage is driven by a well-behaved local sharpness constant of a homogeneous system. Nevertheless, the refined result in \cite{lu2023geometry} is only developed for the vanilla PDHG; thus, it has a non-accelerated rate, and it is unclear how to extend this refined analysis to raPDHG. 

In addition, for infeasibility detection, it was shown that vanilla PDHG can recover the infeasibility certificate for LP with an eventual linear convergence rate under a non-degeneracy condition~\cite{applegate2024infeasibility}. However, such a non-degeneracy condition usually does not hold for practical LPs (it is also called an ``irrepresentable'' condition in the related literature~\cite{fadili2018sensitivity,fadili2019model}). Even under the non-degeneracy condition, the linear convergence of vanilla PDHG is non-accelerated, and there is no complexity theory of the infeasibility detection for raPDHG.

Indeed, the fundamental reason for the lack of an accelerated two-stage rate and accelerated infeasibility detection rate of raPDHG is as follows: there is no guarantee for the convergence of fixed point residual at the average iterates, which is essential for the refined analysis in \cite{lu2023geometry} and the infeasibility detection in \cite{applegate2024infeasibility}. 
To address the above issues, this paper visits the following two natural questions:
\begin{itemize}
    \item Is there a matrix-free primal-dual algorithm with an accelerated two-stage behavior for solving feasible LP?
    \item Is there a matrix-free primal-dual algorithm with an accelerated linear convergence rate for recovering infeasibility certificates for infeasible LP?
\end{itemize}

We provide affirmative answers to the above questions by proposing a new algorithm, which we dubbed restarted Halpern PDHG (rHPDHG), for solving LP. Particularly, the contributions of the paper can be summarized as follows:
\begin{itemize}
    \item We propose restarted Halpern PDHG (rHPDHG) for solving LP \eqref{eq:lp} (Section \ref{sec:restart-halpern}), and show that it achieves the same accelerated linear convergence rate as \cite{applegate2023faster} on feasible and bounded LP as raPDHG when using the global Hoffman constant (Section \ref{sec:accelerated_rate}).
    \item We present an accelerated two-stage convergence complexity for rHPDHG (Section \ref{sec:two_stage}), i.e., rHPDHG has an improved complexity to identify the active variables and an accelerated eventual linear convergence rate, in contrast to the non-accelerated rate of vanilla PDHG~\cite{lu2023geometry}.
    \item We show that rHPDHG can recover infeasibility certificates with an accelerated linear convergence rate (Section \ref{sec:infeas}), in contrast to the non-accelerated rate of vanilla PDHG~\cite{applegate2024infeasibility}.
    \item We discuss an extension of rHPDHG by adding reflection (dubbed $\mathrm{r^2HPDHG}$) and discuss how reflection can improve the theoretical guarantees of rPHDG by a factor of 2 in theory.
    \item We built up a new LP solver on GPU based on rHPDHG/$\mathrm{r^2HPDHG}$, which we call HPDLP, and the numerical experiments on the MIPLIB benchmark set demonstrate a similar performance as cuPDLP. This showcases the strong numerical performance of rHPDHG/$\mathrm{r^2HPDHG}$ for LP, by noticing that cuPDLP has comparable numerical performance as state-of-the-art LP solvers such as Gurobi and COPT~\cite{lu2023cupdlp,lu2023cupdlpc}. 
\end{itemize}

\subsection{Related literature}

{\bf Linear programming.} Linear programming (LP)~\cite{dantzig2002linear,luenberger1984linear} is a fundamental tool in operations research and computer science, with numerous practical applications~\cite{anderson2000hotel,bowman1956production,boyd2004convex,charnes1954stepping,hanssmann1960linear,liu2008choice,manne1960linear}. The optimization community has been working on accelerating and scaling up LP since the 1940s, leading to extensive research in both academia and industry. The leading methods for solving LP are the simplex methods~\cite{dantzig1998linear} and interior-point methods~\cite{karmarkar1984new}. Leveraging these methods, commercial solvers like Gurobi and COPT, along with open-source solvers such as HiGHS, offer reliable solutions with high accuracy.

{\bf FOM solvers for LP.} Recent research has increasingly focused on using first-order methods for solving large-scale linear programming due to their low iteration costs and parallelization capabilities. Here is an overview of several FOM solvers:
\begin{itemize}
\item \href{https://github.com/google/or-tools/tree/stable/ortools/pdlp}{PDLP}~\cite{applegate2021practical,lu2023cupdlp,lu2023cupdlpc} is a general-purpose large-scale LP solver built upon raPDHG algorithm~\cite{applegate2023faster}, with many practical algorithmic enhancements. The CPU implementation  (open-sourced through \href{https://developers.google.com/optimization}{Google OR-Tools}) is demonstrated to have superior performance than other FOM solver on LP~\cite{applegate2021practical}, while the GPU implementation cuPDLP (\href{https://github.com/jinwen-yang/cuPDLP.jl}{cuPDLP.jl} and \href{https://github.com/COPT-Public/cuPDLP-C}{cuPDLP-C}) has shown comparable performance to commercial LP solver such as Gurobi and COPT~\cite{lu2023cupdlp,lu2023cupdlpc}.

\item\href{https://github.com/leavesgrp/ABIP}{ABIP}~\cite{lin2021admm, deng2022new} is a matrix-free IPM solvers for conic programming. ABIP solves LP as a special case of cone program. It adopts the homogenous self-dual embedding via IPM and utilizes multiple ADMM iterations instead of one Newton step to approximately minimize the log-barrier penalty function.  ABIP+~\cite{deng2022new}, an enhanced version of ABIP, includes many practical enhancements on top of ABIP. 

\item ECLIPSE~\cite{basu2020eclipse} is a distributed LP solver that leverages accelerated gradient descent to solve a smoothed dual form of LP. ECLIPSE is designed specifically to solve large-scale LPs with certain decomposition structures arising from web applications.

\item \href{https://github.com/cvxgrp/scs}{SCS}~\cite{o2016conic,o2021operator} is designed to solve convex cone programs based on ADMM. The computational bottleneck of ADMM-based methods is solving a linear system with similar forms every iteration. It has an indirect option using conjugate gradient methods to solve the system and a GPU implementation based on it. SCS can solve LPs as special cases and support solving linear systems through direct factorization or conjugate gradient methods, with trade-offs in scalability.
\end{itemize}

{\bf Primal-dual hybrid gradient (PDHG).} Originally developed for image processing applications, the Primal-Dual Hybrid Gradient (PDHG) method was introduced in a series of foundational studies~\cite{chambolle2011first, condat2013primal, esser2010general, he2012convergence, zhu2008efficient}. The first convergence proof for PDHG was provided in~\cite{chambolle2011first}, showing a sublinear ergodic convergence rate of $O(1/k)$ for convex-concave primal-dual problems. Later research simplified the analysis of this convergence rate~\cite{chambolle2016ergodic, lu2023unified}. More recent findings have indicated that PDHG's last iterates can achieve linear convergence under certain mild regularity conditions, applicable to various fields including linear programming~\cite{fercoq2022quadratic, lu2022infimal}. The methodology has been extended through various adaptations, including adaptive and stochastic versions of PDHG~\cite{goldstein2015adaptive, malitsky2018first, pock2011diagonal, vladarean2021first, alacaoglu2022convergence, chambolle2018stochastic}. Additionally, research has shown that PDHG can be equated to Douglas-Rachford Splitting with a linear transformation~\cite{liu2021acceleration, o2020equivalence,lu2023unified}.

{\bf Halpern iteration.} The Halpern iteration was first proposed in \cite{halpern1967fixed} for solving fixed point problems and it attracted recent attention with its application in accelerating minimax optimization~\cite{ryu2019ode,yoon2021accelerated,diakonikolas2020halpern,cai2022accelerated,tran2022connection,tran2021halpern,cai2022stochastic}. The asymptotic convergence was firstly established in \cite{xu2002iterative,wittmann1992approximation} and the convergence rate is derived in early works \cite{leustean2007rates,kohlenbach2011quantitative}. Later, an improved rate $O(1/k^2)$ of Halpern iteration is derived in \cite{sabach2017first,lieder2021convergence}. The exact tightness is established in \cite{park2022exact,kim2021accelerated} by constructing a matching complexity lower bound. Halpern iteration is also related to the anchoring techniques developed in \cite{ryu2019ode,yoon2021accelerated}.

{\bf Complexity of FOM on LP.} A significant limitation of first-order methods (FOMs) on LP is their slow tail convergence. The inherent lack of strong convexity in LP leads to merely sublinear convergence rates when classical FOM results are applied~\cite{nesterov2013introductory, beck2017first, ryu2022large}, making it challenging to achieve the high-precision solutions typically expected within reasonable time. However, it turns out that LP possesses additional structural properties that can enhance FOM convergence from sublinear to linear. For instance, a modified version of the ADMM achieves linear convergence for LP~\cite{eckstein1990alternating}. More recently, in~\cite{applegate2023faster}, a sharpness condition is introduced that is satisfied by LP, and a restarted scheme is proposed for various primal-dual methods to solve LP. It turns out that the restarted variants of many FOMs, such as PDHG, extra-gradient method (EGM) and ADMM, can achieve global linear convergence, and such linear convergence is the optimal rate for solving LP, namely, there is a worst-case LP instance such that no FOMs can achieve better than such linear rate. Additionally, it has been shown that PDHG, even without restarting, attains linear convergence, albeit at a slower rate than the restarted methods~\cite{lu2022infimal}. The convergence rates discussed in \cite{applegate2023faster,lu2022infimal} depend on the global Hoffman constant related to the KKT system of LP, which is generally considered overly conservative and not reflective of actual algorithmic performance. Motivated, a refined complexity theory for PDHG in LP has been developed that does not rely on the global Hoffman constant, providing a more accurate depiction of its two-stage (identification + local convergence) behavior~\cite{lu2023geometry}. In~\cite{xiong2023computational,xiong2023relation,xiong2024role}, two purely geometry-based condition measures of LP are introduced to characterize the behavior of raPDHG. The instance-independent complexity of raPDHG for solving totally unimodular LP is derived in \cite{hinder2023worst}.

{\bf Infeasibility detection.} Research on detecting infeasibility in FOMs for convex optimization has predominantly centered on the Alternating Direction Method of Multipliers (ADMM), or equivalently, Douglas-Rachford Splitting. It was demonstrated that the iterates diverge when a solution is non-existent~\cite{eckstein1992douglas}. More recently, it was shown that the infimal displacement vector in ADMM can serve as a certificate of infeasibility for convex quadratic problems with conic constraints, suggesting the use of iterate differences as a diagnostic tool~\cite{banjac2019infeasibility}. Moreover, it turns out that the difference of iterates under any firmly nonexpansive operator converges at a sublinear rate and a multiple-run ADMM approach is introduced to detect infeasibility~\cite{liu2019new}, especially designed for complex scenarios beyond LP. Additionally, recent studies in~\cite{park2023accelerated} has explored the optimal accelerated rate for infeasibility detection for general problems. The behavior of vanilla PDHG on LP was studied in \cite{applegate2024infeasibility}. It turns out that the infimal displacement vector is indeed a certificate of infeasibility, namely, satisfying Farkas lemma. The difference of iterates if PDHG is demonstrated to exhibit eventual linear convergence under non-degeneracy condition. The results of PDHG were extended to quadratic programming and conic programming in \cite{jiang2023range}.

\subsection{Notation}
Denote $\|\cdot\|_2$ the Euclidean norm for vector and spectral norm for matrix. Without loss of generality, we assume the primal and dual step-size in PDHG are equal, i.e., $\tau=\sigma=:\eta\leq \frac{1}{2\|A\|_2}$. To ease the notation, we use the norm $\|\cdot\|=\langle\cdot,P_\eta\cdot\rangle$ to denote the canonical norm of PDHG, where $P_\eta=\begin{pmatrix}
    \frac{1}{\eta}I & -A^\top \\ -A & \frac{1}{\eta}I
\end{pmatrix}$. The distance from a point to a set with respect to norm $\|\cdot\|$ is denoted as $\mathrm{dist}(x,\mathcal U):=\argmin_{u\in\mathcal U}\|x-u\|$. We use big O notation to characterize functions according to their growth rate, particularly, $f(x)= O(g(x))$ means that for sufficiently large $x$, there exists constant $C$ such that $f(x)\leq Cg(x)$ while $f(x)=\widetilde O(g(x))$ suppresses the log dependence, i.e., $f(x)=O(g(x)\log(g(x))$. For a vector $x\in \mathbb{R}^n$ and any set $S\subseteq \{1,2,...,n\}$, denote $x_S=(x_i)_{i\in S}\in \mathbb{R}^{|S|}$ as the subvector with the corresponding coordinates of $x$. We use $z=(x,y)$ to represent the primal-dual solution pair. We use $T(z)=(\tilde x,\tilde y)$ to denote the operator for one PDHG iteration from $z$ (see Section \ref{sec:restart-halpern} for a formal definition). We denote $(T(z)_x)_S=\tilde x_S$ and $T(z)_y=\tilde y$ as the corresponding primal and dual variables after applying the PDHG operator to $z$. Denote $\mathrm{Fix}(T):=\{z\mid z=T(z)\}$ the set of fixed points of the PDHG operator $T$. We call $\|z-T(z)\|$ the fixed point residual of operator $T$.

\section{Restarted Halpern PDHG for LP}\label{sec:restart-halpern}
In this section, we propose our method, restarted Halpern PDHG (rHPDHG, Algorithm \ref{alg:hpdhg-restart}), for solving \eqref{eq:minmax}, and discuss its connection and difference with restarted average PDHG (raPDHG), which is the base algorithm used in the LP solver PDLP.

We consider the standard form LP
\begin{align}\label{eq:lp}
    \begin{split}
        \min_{x\geq 0}\; c^\top x \quad \mathrm{s.t.}\; Ax=b \ ,
    \end{split}
\end{align}
and its primal-dual formulation
\begin{align}\label{eq:minmax}
    \begin{split}
        \min_{x\geq 0}\max_{y}\; c^\top x+y^\top Ax-b^\top y \ .
    \end{split}
\end{align}
Denote $z=(x,y)$ the primal-dual pair. The update rule of PDHG~\cite{chambolle2011first} for solving \eqref{eq:minmax}, which we denote as $z^{k+1}=T(z^k)$,  is given as
\begin{equation}\label{eq:pdhg}
    \begin{cases}
        x^{k+1}\leftarrow \text{proj}_{\mathbb R^n_+}(x^k-\tau A^\top y^k-\tau c) \\ y^{k+1}\leftarrow y^k+\sigma A(2x^{k+1}-x^k)-\sigma b\ .
    \end{cases}
\end{equation}
We call $T$ the operator for a PDHG iteration. Halpern method is a scheme to accelerate general operator splitting methods that have been extensively studied for solving minimax optimization problems~\cite{ryu2019ode,yoon2021accelerated,diakonikolas2020halpern,cai2022accelerated,tran2022connection,tran2021halpern,cai2022stochastic}. Upon vanilla PDHG, Halpern PDHG~\cite{halpern1967fixed} takes a weighted average between the PDHG step of current iterate and the initial point.  Specifically, the update rule of Halpern PDHG is as follows 
\begin{equation}\label{eq:hpdhg}
    z^{k+1}=\text{H-PDHG}(z^k;z^0):=\frac{k+1}{k+2}T(z^k)+\frac{1}{k+2}z^0 \ .
\end{equation}

\begin{algorithm}
\caption{Restarted Halpern PDHG $rHPDHG$ for \eqref{eq:minmax}}
\label{alg:hpdhg-restart}
\SetKwInOut{Input}{Input}
\Input{Initial point $z^{0,0}$, outer loop counter $n\leftarrow 0$.}

\Repeat{\upshape $z^{n+1,0}$ convergence}{
  initialize the inner loop counter $k\leftarrow0$;\\
  \Repeat{\upshape restart condition holds}{
    $z^{n,k+1}\leftarrow \text{H-PDHG}(z^{n,k};z^{n,0})$;
  }
  initialize the initial solution $z^{n+1,0}\leftarrow T(z^{n,k})$;\\ 
  $n\leftarrow n+1$;
}
\end{algorithm}

Algorithm \ref{alg:hpdhg-restart} outlines our nested-loop restarted Halpern PDHG. It begins with the initialization at $z^{0,0}$. In each outer loop iteration, we continue to execute the Halpern PDHG until a specified restart condition is met (these conditions are described in later sections). During the $k$-th inner loop iteration of the $n$-th outer loop, we invoke Halpern PDHG to obtain the next iterate $z^{n,k+1}$. The next outer loop is initiated at the point of one single PDHG step from the last iterate of the previous outer loop.

rHPDHG uses Halpern PDHG as its base algorithm and restarts at a single PDHG step of the current iterate. In contrast, raPDHG~\cite{applegate2023faster} keeps running the vanilla PDHG in each epoch and restarts at the average iterates of the epoch. 
It turns out that there are connections between the two schemes. Particularly, they share the same trajectory for solving unconstrained bilinear problems as illustrated in Proposition \ref{prop:equivalence}:
\begin{equation}\label{eq:unconstrined}
    \min_x\max_y\ c^\top x-y^\top Ax+b^\top y \ .
\end{equation}
 
\begin{prop}\label{prop:equivalence}
    Consider solving the unconstrained bilinear problem \eqref{eq:unconstrined}. Denote $\{\tilde z^k\}$ and $\{z^k\}$ the iterates of vanilla PDHG \eqref{eq:pdhg} and Halpern PDHG \eqref{eq:hpdhg}, respectively. Suppose $\tilde z^0=z^0$. Denote $\bar z^k:=\frac{1}{k+1}\sum_{i=0}^k \tilde z^i$ the average iterates of vanilla PDHG. Then it holds for any $k\geq 0$ that
    \begin{equation*}
        z^k=\bar z^k \ .
    \end{equation*}
\end{prop}
\begin{proof}
    Denote matrix $Q:=\begin{pmatrix}
            I & -\eta A^\top\\ \eta A & I-\eta^2 AA^\top
        \end{pmatrix}$ and vector $s:=\begin{pmatrix}
            -\eta c \\ -\eta b-2\eta^2 Ac
        \end{pmatrix}$.
    Note that the $k$-th iterate of PDHG $\tilde z^k$ follows the dynamical system
    \begin{equation}\label{eq:bilinear-1}
        \tilde z^{k}=\begin{pmatrix}
            I & -\eta A^\top\\ \eta A & I-\eta^2 AA^\top
        \end{pmatrix}\tilde z^{k-1}+\begin{pmatrix}
            -\eta c \\ -\eta b-2\eta^2 Ac
        \end{pmatrix}=:Q\tilde z^{k-1}+s \ ,
    \end{equation}
    and thus the average iterates can be written as 
    \begin{equation}\label{eq:bilinear-2}
        \bar z^k = \frac{1}{k+1}\sum_{i=0}^k \tilde z^i=\frac{1}{k+1}\pran{\sum_{i=0}^kQ^i}z^0+\frac{1}{k+1}\pran{\sum_{i=0}^k(k-i)Q^i}s \ .
    \end{equation}
The proof follows from induction. Suppose $z^k=\bar z^k$ (which holds for $k=0$). Then we have
\begin{equation*}
\begin{aligned}
    z^{k+1}&=\frac{k+1}{k+2}(Qz^k+s)+\frac{1}{k+2}z^0 = \frac{1}{k+2}\pran{\sum_{i=0}^kQ^{i+1}}z^0+\frac{1}{k+2}\pran{\sum_{i=0}^k(k-i)Q^{i+1}+(k+1)I}s+\frac{1}{k+2}z^0 \\
    & = \frac{1}{k+2}\pran{\sum_{i=0}^{k+1}Q^{i}}z^0+\frac{1}{k+2}\pran{\sum_{i=0}^{k+1}(k+1-i)Q^{i}}s=\bar z^{k+1} \ ,
\end{aligned}
\end{equation*}
where the first equality utilizes \eqref{eq:bilinear-1} and the second one uses \eqref{eq:bilinear-2}. By induction, we prove Halpern PDHG is equivalent to average PDHG on unconstrained bilinear problem.
\end{proof}

This connection may provide some intuitions on the effectiveness of rHPDHG for LP by noticing raPDHG is an optimal algorithm for LP~\cite{applegate2023faster}. Despite such connections, Halpern PDHG exhibits different trajectories of iterates with average PDHG for generic LP~\eqref{eq:lp}, as illustrated in an example in Figure \ref{fig:bilinear}. Actually, the difference between LP~\eqref{eq:lp} and unconstrained bilinear problem~\eqref{eq:unconstrined} is just the projection onto the positive orthant. In other words, the difference between raPDHG and rHPDHG inherently comes from the projection step for the positive orthant constraint $x\ge 0$. 
\begin{figure}[ht!]
\centering
\begin{subfigure}{0.5\textwidth}
  \centering
  \includegraphics[width=\textwidth]{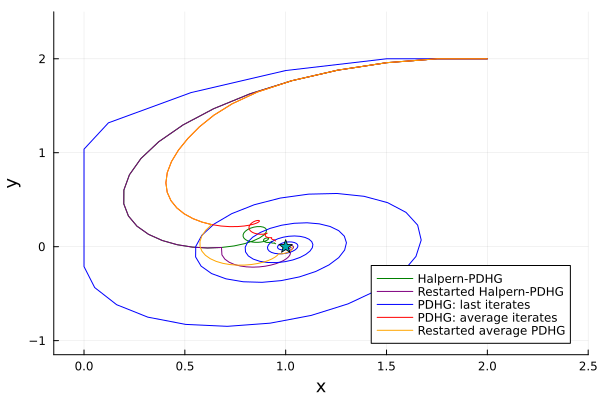}
  \caption{Comparison of trajectories}
\end{subfigure}%
\begin{subfigure}{0.5\textwidth}
  \centering
  \includegraphics[width=\textwidth]{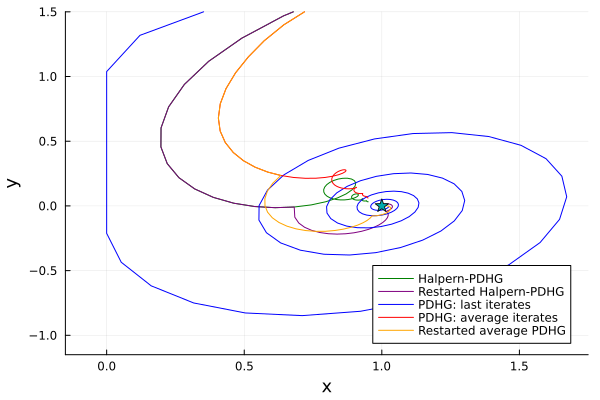}
  \caption{Zoom-in around optimal solution}
\end{subfigure}
\caption{Comparison of trajectories on solving $\min_{x\geq 0}\max_y\ (x-1)y$.}
\label{fig:bilinear}
\end{figure}

\section{Accelerated linear rate of rHPDHG for LP}\label{sec:accelerated_rate}

In this section, we present the accelerated linear convergence rate of rHPDHG for LP. Similar to raPDHG~\cite{applegate2023faster}, we introduce a fixed frequency restarting scheme and an adaptive restarting scheme and show that both restarting schemes can achieve the accelerated linear convergence rate, similar to raPDHG, for feasible and bounded LP. Compared to \cite{applegate2023faster}, there are two major differences: (i) we define sharpness condition using the fixed point residual instead of normalized duality gap in \cite{applegate2023faster}; (ii) we utilize fixed point residual to determine adaptive restarting instead of normalized duality gap. 

We first introduce a sharpness condition of \eqref{eq:lp}. This condition has been extensively used recently for first-order methods to achieve linear rates on solving LP~\cite{applegate2023faster}. Sharpness essentially measures the growth of sub-gradient with respect to distance to optimality. More specifically, the following sharpness property holds for PDHG operator on solving LP \eqref{eq:lp}.
\begin{prop}\label{prop:sharp}
Consider linear programming ~\eqref{eq:lp} and denote $T(z)$ as one PDHG iteration from $z$~\eqref{eq:pdhg}. For any $R>0$, there exists a constant $\alpha_\eta>0$ such that for any $z$ such that $\|T(z)\|_2\leq R$, it holds that
    \begin{equation*}
        \alpha_\eta\mathrm{dist}(z,\mathcal Z^*)\leq \left\|z-\PDHG(z)\right\| \ ,
    \end{equation*}
    where $\mathcal Z^*$ is the set of optimal solutions.
\end{prop}
\begin{proof}
    We know from~\cite{applegate2023faster,lu2022infimal,fercoq2022quadratic}  that for any $R>0$, there exists an $\alpha>0$ such that for any $z$ with $\|z\|_2\leq R$, it holds that
    \begin{equation}\label{eq:eucildean-sharp}
        \alpha\mathrm{dist}_2(z,\mathcal Z^*)\leq \mathrm{dist}_2(0,\mathcal F(z)) \ ,
    \end{equation}
    where $\mathcal F(z)=\mathcal F(x,y)=\begin{pmatrix}
        c+A^\top y+\partial \iota_{\mathbb R_+^n}(x) \\ b-Ax
    \end{pmatrix}$ is the sub-differential to \eqref{eq:minmax}. 
   Then denote $\alpha_\eta = \frac{\eta\alpha}{\eta\alpha+2}$ and we have
    \begin{equation*}
        \begin{aligned}
            \mathrm{dist}(z,\mathcal Z^*) &\leq \mathrm{dist}(T(z),\mathcal Z^*) + \|z-T(z)\|\leq \sqrt{\frac 2\eta} \mathrm{dist}_2(T(z),\mathcal Z^*)+ \|z-T(z)\|\\
            &\leq \sqrt{\frac2\eta}\frac{1}{\alpha}\mathrm{dist}_2(0,\mathcal F(T(z)))+ \|z-T(z)\| \leq \frac{2}{\eta\alpha}\mathrm{dist}_{P_\eta^{-1}}(0,\mathcal F(T(z)))+ \|z-T(z)\|\\
            & \leq \frac{2}{\eta\alpha}\|z-T(z)\|+ \|z-T(z)\| = \pran{1+\frac{2}{\eta\alpha}}\|z-T(z)\|= \frac{1}{\alpha_\eta}\|z-T(z)\| \ ,
        \end{aligned}
    \end{equation*}
    where the third inequality uses \eqref{eq:eucildean-sharp} at $T(z)$ and the last inequality utilizes  $P_\eta(z-T(z))\in \mathcal F(T(z))$ by update rule of PDHG \eqref{eq:pdhg}.
\end{proof}

Next, we propose two restart schemes for Algorithm \ref{alg:hpdhg-restart}:
\begin{itemize}
    \item \textbf{Fixed frequency restart for feasible LP}: Suppose we know the sharpness constant $\alpha_\eta$. Under this approach, we interrupt the inner loop and initiate the outer loop anew at a predetermined frequency $k^*$. Specifically, we restart the algorithm when
    \begin{equation}\label{eq:fixed-restart}
        k\geq \frac{2e}{\alpha_\eta}=k^* \ .
    \end{equation}
    \item \textbf{Adaptive restart for feasible LP}: In this approach, we initiate a restart when there is significant decay in the difference between iterates $z^{n,k}$ and $T(z^{n,k})$. This adaptive restart mechanism does not require an estimate of $\alpha_\eta$ which is almost inaccessible in practice~\cite{pena2021new}. Specifically, we restart the algorithm if
    \begin{equation}\label{eq:adaptive-restart}
    \begin{cases}
        \left\|z^{n,k}-\PDHG(z^{n,k})\right\|\leq \frac 1e \left\|z^{n,0}-\PDHG(z^{n,0})\right\|, & n\geq 1\\
        k>\tau^0,& n=0
    \end{cases}
    \end{equation}
\end{itemize}

The main results of this section, Theorem \ref{thm:fixed-restart} and Theorem \ref{thm:adaptive-restart}, present the linear convergence rate of restarted Halpern PDHG with two restart schemes stated above. 
\begin{thm}[Fixed frequency restart for feasible LP]\label{thm:fixed-restart}
For a feasible and bounded LP~\eqref{eq:lp}, consider $\{z^{n,k}\}$ the iterates of restarted Halpern PDHG (Algorithm \ref{alg:hpdhg-restart}) with fixed frequency restart scheme, namely, we restart the outer loop if \eqref{eq:fixed-restart} holds. Denote $\mathcal Z^*$ is the set of optimal solutions. Then it holds for any $n\geq 0$ that
    \begin{equation*}
        \mathrm{dist}(z^{n+1,0},\mathcal Z^*)\leq \pran{\frac 1e}^{n+1} \mathrm{dist}(z^{0,0},\mathcal Z^*) \ .
    \end{equation*}
\end{thm}

\begin{thm}[Adaptive restart for feasible LP]\label{thm:adaptive-restart}
For a feasible and bounded LP~\eqref{eq:lp}, consider $\{z^{n,k}\}$ the iterates of restarted Halpern PDHG (Algorithm \ref{alg:hpdhg-restart}) with adaptive restart scheme, namely, we restart the outer loop if \eqref{eq:adaptive-restart} holds. Denote $\mathcal Z^*$ is the set of optimal solutions. Then it holds for any $n\geq 0$ that
\begin{enumerate}
    \item[(i)] The restart length $\tau^n$ is upper bounded by $k^*$,
    \begin{equation*}
        \tau^n\leq \frac{2e}{\alpha_n}\leq \left\lceil{\frac{2e}{\alpha_\eta}}\right\rceil=:k^* \ ,
    \end{equation*}
    where $\alpha_n:=\frac{\|z^{n,0}-\PDHG(z^{n,0})\|}{\mathrm{dist}(z^{n,0},\mathcal Z^*)}\geq \alpha_\eta$.
    \item[(ii)] The distance to optimal set decays linearly,
    \begin{equation*}
        \mathrm{dist}(z^{n+1,0},\mathcal Z^*) \leq \pran{\frac 1e}^{n+1}\frac{2e}{\alpha_\eta (\tau^0+1)}\mathrm{dist}(z^{0,0},\mathcal Z^*) \ .
    \end{equation*}
\end{enumerate}
\end{thm}
\begin{rem}
    Theorem \ref{thm:fixed-restart} and Theorem \ref{thm:adaptive-restart} imply the $O\pran{\frac{1}{\alpha_\eta}\log\frac{1}{\epsilon}}$ for fixed frequency scheme and $\widetilde O\pran{\frac{1}{\alpha_\eta}\log\frac{1}{\epsilon}}$ for adaptive restart scheme respectively, to achieve an $\epsilon$-accuracy solution of \eqref{eq:lp}. This shows both algorithms achieve (nearly) optimal rate among a class of matrix-free algorithms for solving LP~\cite{applegate2023faster}.
\end{rem}

In the rest of this section, we show the proof of Theorem \ref{thm:fixed-restart} and Theorem \ref{thm:adaptive-restart}. Recall that we assume step-size $\eta\leq \frac{1}{2\|A\|_2}$ throughout the paper. While our results hold for step-size $\eta<\frac{1}{\|A\|_2}$, assuming $\eta\leq \frac{1}{2\|A\|_2}$ makes it easy to convert between PDHG canonical norm $\|\cdot\|$ and $\ell_2$ norm $\|\cdot\|_2$, as stated in Lemma \ref{lem:change-of-norm}.
\begin{lem}\label{lem:change-of-norm}
Suppose step-size $\eta\leq \frac{1}{2\|A\|_2}$. Then it holds for any $z\in \mathbb R^{m+n}$ that
    \begin{equation*}
        \sqrt{\frac{1}{2\eta}}\|z\|_2\leq \|z\|\leq \sqrt{2\eta}\|z\|_2 \ .
    \end{equation*}
\end{lem}
We start with a bucket of basic properties of restarted Halpern PDHG iterates which is collected in the following lemma.
\begin{lem}\label{lem:iterates-feas}
For a feasible and bounded LP~\eqref{eq:lp}, consider $\{z^{n,k}\}$ the iterates generated by restarted Halpern PDHG (Algorithm \ref{alg:hpdhg-restart}) with either fixed frequency or adaptive restart. Then, the following properties hold.
\begin{enumerate}
    \item[(i)] For any $\tilde z^*\in\mathcal Z^*$, $\|z^{n+1,0}-\tilde z^*\|\leq \|z^{n,k}- \tilde z^*\|\leq \|z^{n,0}-\tilde z^*\|$.
    \item[(ii)] There exists a constant $R>0$ such that for any $n$ and $k$, $\|z^{n,k}\|_2\leq R$.
    \item[(iii)] There exists an optimal solution $z^*$ such that $\lim_{n\rightarrow \infty}z^{n,0}=z^*$.
    \item[(iv)] For any $\tilde z^*\in\mathcal Z^*$, $\|z^*-\tilde z^*\|\leq \|z^{n,k}-\tilde z^*\|$ and $\|z^{n,k}-z^*\|\leq 2\mathrm{dist}(z^{n,k},\mathcal Z^*)$.
\end{enumerate}
\end{lem}
\begin{proof}
(i) We prove by induction. Suppose $\|z^{n,k-1}-\tilde z^*\|\leq \|z^{n,0}-\tilde z^*\|$. Then
\begin{equation*}
        \begin{aligned}
            \|z^{n,k}-\tilde z^*\|&=\left\|\frac{k}{k+1} \PDHG(z^{n,k-1})+\frac{1}{k+1} z^{n,0}-\tilde z^*\right\|\leq \frac{k}{k+1} \|\PDHG(z^{n,k-1})-\tilde z^*\|+\frac{1}{k+1} \|z^{n,0}-\tilde z^*\|\\
            &\leq \frac{k}{k+1} \|z^{n,k-1}-\tilde z^*\|+\frac{1}{k+1} \|z^{n,0}-\tilde z^*\| \leq \|z^{n,0}-\tilde z^*\| \ ,
        \end{aligned}
    \end{equation*}
    where the second inequality is due to the nonexpansiveness of $T$, and the last inequality follows from the induction assumption. In addition, $\|z^{n+1,0}-\tilde z^*\|=\|\PDHG(z^{n,\tau^n})-\tilde z^*\|\leq \|z^{n,\tau^n}-\tilde z^*\|\leq \|z^{n,0}-\tilde z^*\|$.

(ii) Select an optimal solution $\tilde z^*\in\mathcal Z^*$. Denote $R=2(\|z^{0,0}-\tilde z^*\|_2+\|\tilde z^*\|_2)$. Then it holds for any $n$ and $k$ that
\begin{equation*}
\begin{aligned}
    \|z^{n,k}\|_2&\leq \sqrt{2\eta} \|z^{n,k}\|\leq \sqrt{2\eta}\|z^{n,k}- \tilde z^*\|+\sqrt{2\eta}\|\tilde z^*\|\leq \sqrt{2\eta}\|z^{n,0}-\tilde z^*\|+\sqrt{2\eta}\|\tilde z^*\|\\
    &\leq \sqrt{2\eta}\|z^{0,0}-\tilde z^*\|+\sqrt{2\eta}\|\tilde z^*\| \leq 2(\|z^{0,0}-\tilde z^*\|_2+\|\tilde z^*\|_2)=R\ ,
\end{aligned}
\end{equation*}
where the third and the fourth inequalities utilize (i).

(iii) Since $z^{n,0}$ is bounded, there exists an optimal solution $z^*$ and a subsequence $n_k$ such that $z^{n_k,0}\rightarrow z^*$. Note that $\|z^{n,0}-z^*\|$ is non-increasing and we know $z^{n,0}\rightarrow z^*$. 

(iv) For any $n'\geq n+1$ and any $\tilde z^*\in\mathcal Z^*$,
\begin{equation*}
    \|z^{n',0}-\tilde z^*\|\leq \|z^{n+1,0}-\tilde z^*\|\ .
\end{equation*}
Let $n'\rightarrow \infty$ and thus $z^{n',0}\rightarrow z^*$. We have $\|z^*-\tilde z^*\|\leq \|z^{n+1,0}-\tilde z^*\|\leq \|z^{n,k}-\tilde z^*\|$. In addition, 
\begin{equation*}
    \|z^{n,k}-z^*\|\leq \|z^{n,k}-P_{\mathcal Z^*}(z^{n,k})\|+\|z^*-P_{\mathcal Z^*}(z^{n,k})\|\leq 2\|z^{n,k}-P_{\mathcal Z^*}(z^{n,k})\| = 2\mathrm{dist}(z^{n,k},\mathcal Z^*) \ ,
\end{equation*}
where $P_{\mathcal Z^*}(z)$ is the projection in the corresponding norm of $z$ to the set $Z^*$.
\end{proof}

The following lemma shows the sublinear convergence of Halpern PDHG.
\begin{lem}[{\cite[Theorem 2.1 and its proof (page 4)]{lieder2021convergence}}]\label{lem:sublinear-feas}
    Consider $\{z^{k}\}$ the iterates of Halpern PDHG on solving feasible \eqref{eq:lp}. Denote $z^*\in \mathcal Z^*:=\{z^*\mid \PDHG(z^*)-z^*=0\}$. Then it holds for any $k\geq 1$ that
    \begin{equation*}
        \left\|z^k-\PDHG(z^k)\right\|\leq \frac{2}{k+1}\mathrm{dist}(z^0,\mathcal Z^*), \ \mathrm{and} \ \left\|z^k-\PDHG(z^k)\right\|\leq \frac{2}{k}\left\|z^k-z^0\right\|
    \end{equation*}
\end{lem}

Now we are ready to prove Theorem \ref{thm:fixed-restart} and Theorem \ref{thm:adaptive-restart}.
\begin{proof}[Proof of Theorem \ref{thm:fixed-restart}]
From part (ii) of Lemma \ref{lem:iterates-feas}, we know the iterates $z^{n,k}$ are bounded. Thus it holds that
    \begin{equation*}
        \begin{aligned}
            \mathrm{dist}(z^{n+1,0},\mathcal Z^*)& \leq \frac{1}{\alpha_\eta}\left\|z^{n+1,0}-\PDHG(z^{n+1,0})\right\| = \frac{1}{\alpha_\eta}\left\|\PDHG(z^{n,k^*})-\PDHG(\PDHG(z^{n,k^*}))\right\|\leq  \frac{1}{\alpha_\eta}\left\|z^{n,k^*}-\PDHG(z^{n,k^*})\right\|\\
            &\leq \frac{2}{\alpha_\eta(k^*+1)}\mathrm{dist}(z^{n,0},\mathcal Z^*)\leq \frac 1e \mathrm{dist}(z^{n,0},\mathcal Z^*) \leq ...\leq \pran{\frac 1e}^{n+1} \mathrm{dist}(z^{0,0},\mathcal Z^*) \ ,
        \end{aligned}
    \end{equation*}
    where the first inequality utilized Proposition \ref{prop:sharp} and the first equality follows from the definition of $z^{n+1,0}=\PDHG(z^{n,k^*})$. The second and third inequality use non-expansiveness of PDHG operator $\PDHG$ and Lemma \ref{lem:sublinear-feas}, respectively. The fourth inequality is due to the choice of restart frequency $k^*$. The proof is completed by recursion.
\end{proof}

\begin{proof}[Proof of Theorem \ref{thm:adaptive-restart}]
(i) Due to the restart condition \eqref{eq:adaptive-restart}, it holds for any $n\geq 0$ with $k=\tau^n-1$ that
\begin{equation*}
    \frac 1e\|z^{n,0}-\PDHG(z^{n,0})\|<\|z^{n,\tau^n-1}-\PDHG(z^{n,\tau^n-1})\|\leq \frac{2}{\tau^n}\mathrm{dist} (z^{n,0}, \mathcal Z^*) \ ,
\end{equation*}
where the last inequality follows from Lemma \ref{lem:sublinear-feas}. Thus we have
\begin{equation*}
    \tau^n\leq 2e\frac{\mathrm{dist}(z^{n,0},\mathcal Z^*)}{\|z^{n,0}-\PDHG(z^{n,0})\|}=\frac{2e}{\alpha_n}\leq \frac{2e}{\alpha_\eta}\ ,
\end{equation*}
where the equality is from the definition of $\alpha_n$ and the last inequality uses $\alpha_n\geq \alpha_\eta$ by noticing that $z^{n,k}$ is bounded from part (ii) in Lemma \ref{lem:iterates-feas}.

(ii) It follows from the definition of sharpness and restart condition \eqref{eq:adaptive-restart} that
\begin{equation*}
        \begin{aligned}
            \mathrm{dist}(z^{n+1,0},\mathcal Z^*)& \leq \frac{1}{\alpha_\eta}\left\|z^{n+1,0}-\PDHG(z^{n+1,0})\right\| =\frac{1}{\alpha_\eta}\left\|\PDHG(z^{n,\tau^n})-\PDHG(\PDHG(z^{n,\tau^n}))\right\|\leq \frac{1}{\alpha_\eta}\left\|z^{n,\tau^n}-\PDHG(z^{n,\tau^n})\right\|\\
            & \leq \frac{1}{\alpha_\eta e}\left\|z^{n,0}-\PDHG(z^{n,0})\right\| \leq \ldots \leq \frac{1}{\alpha_\eta e^n}\left\|z^{1,0}-\PDHG(z^{1,0})\right\|\\
            & \leq \pran{\frac 1e}^{n+1}\frac{2e}{\alpha_\eta (\tau^0+1)}\|z^{0,0}-z^*\| = \pran{\frac 1e}^{n+1}\frac{2e}{\alpha_\eta (\tau^0+1)}\mathrm{dist}(z^{0,0},\mathcal Z^*) \ ,
        \end{aligned}
    \end{equation*}
    where the last inequality is from Lemma \ref{lem:sublinear-feas} and the last equality is due to the choice of $z^*=\argmin_{\tilde z^*\in\mathcal Z^*}\|\tilde z^*-z^{0,0}\|$.
\end{proof}

\section{Accelerated two-stage convergence behaviors of rHPDHG}\label{sec:two_stage}
Section \ref{sec:accelerated_rate} presents the global linear convergence of restarted Halpern PDHG. The linear rate depends on the global sharpness constant of LP \eqref{eq:lp}, which in turn depends on the well-known overly conservative Hoffman constant~\cite{hoffman1952approximate,pena2021new,applegate2023faster}. In this sense, this kind of convergence complexity is often too conservative to capture the actual behavior of the algorithm due to its dependence on global sharpness constant (though the rate is tight on worst-case instances). At the same time, it is typical that the algorithm exhibits a two-stage convergence behavior, which is not characterized by the worst-case analysis in Section \ref{sec:accelerated_rate}. In light of these, a refined analysis of vanilla PDHG was presented in \cite{lu2023geometry} that avoids using global sharpness constant and characterizes its two-stage convergence behaviors. However, as shown in \cite{applegate2023faster,lu2022infimal}, the vanilla PDHG is sub-optimal for solving LP, and it is highly challenging to extend the two-stage analysis to raPDHG. The fundamental reason for this difficulty is due to the lack of convergence guarantee on fixed point residual or IDS (introduced in \cite{lu2022infimal}) at the average iteration of PDHG. In this section, we present an accelerated two-stage convergence rate of rHPDHG that avoids the conservative global linear convergence rate, particularly,

\begin{itemize}
    \item In the first stage, Algorithm \ref{alg:hpdhg-restart} identifies active variables. This stage terminates within a finite number of iterations, which can be characterized by a near-degeneracy parameter introduced in \cite{lu2023geometry}. Furthermore, the number of iterations for identification is at the order of square root of vanilla PDHG as in \cite{lu2023geometry}, thus enjoying an accelerated identification rate. 
    
    \item In the second stage, the algorithm essentially solves a system of homogeneous linear inequalities, where the rate of linear convergence is characterized by a local sharpness constant that is much larger than the global system~\cite{pena2024easily,lu2023geometry}. Furthermore, the eventual linear convergence rate is at the order of the square root of vanilla PDHG as in \cite{lu2023geometry}, thus enjoying an accelerated eventual linear convergence rate. 
\end{itemize}

\subsection{Stage I: finite time identification}
In this section, we show that rHPDHG identifies the active basis of the converging optimal solution within a finite number of iterations, i.e., we have finite time identification, even if the problem is degenerate. We call this the identification stage of the algorithm. Furthermore, the duration of this initial stage is characterized by a measure that characterizes the proximity of the non-degenerate components of the optimal solution to degeneracy introduced in \cite{lu2023geometry}.

First, let $z^*=(x^*,y^*)$ be the converging optimal solution of Algorithm \ref{alg:hpdhg-restart} with the adaptive restart scheme. Next, we define an index partition of primal variables:
\begin{mydef}[{\cite[Definition 4]{lu2023geometry}}]\label{def:partition}
For a linear programming~\eqref{eq:minmax} and an optimal primal-dual solution $z^*=(x^*,y^*)$, we define the index partition $(N,B_1,B_2)$ of the primal variable coordinates as
    \begin{align*}
     \ N&=\{i\in [n]: c_i+A_i^\top y^*>0\} \\ \ B_1&=\{i\in [n]: c_i+A_i^\top y^*=0, x_i^*>0\} \\ \  B_2&=\{i\in [n]: c_i+A_i^\top y^*=0, x_i^*=0\} \ .
\end{align*}
\end{mydef}
Following the terminology of \cite{lu2023geometry} (and in analogy to the notions in simplex method), we call the elements in $N$ the non-basic variables, the elements in $B_1$ the non-degenerate basic variables, and the elements in $B_2$ the degenerate basic variables. Furthermore, we call the elements in $N\cup B_1$ the non-degenerate variables and the elements in $B_2$ the degenerate variables. We also denote $B=B_1\cup B_2$ as the set of basic variables. Below we introduce the non-degeneracy metric $\delta$, which essentially measures how close the non-degenerate part of the optimal solution is to degeneracy:

\begin{mydef}[{\cite[Definition 5]{lu2023geometry}}]\label{def:delta}
For a linear programming~\eqref{eq:minmax} and the converging optimal primal-dual solution pair $z^*=(x^*,y^*)$, we define the non-degeneracy metric $\delta$ as:
\begin{equation*}
    \delta := 
    \min\left\{\min_{i\in N} \frac{c_i+A_i^\top y^*}{\|A\|_2},\;\min_{i\in B_1} x_i^*\right\} \ .
\end{equation*}
\end{mydef}
For every optimal solution $z^*$ to a linear programming, the non-degeneracy metric $\delta$ is always strictly positive. The smaller the $\delta$ is, the closer the non-degenerate part of the solution is to degeneracy.

Next, we introduce the sharpness to a homogeneous conic system  $\hat\alpha_{L_1}$ that arises in the complexity of identification. Denote $R=2(\|z^{0,0}-z^*\|_2+\|z^*\|_2)$.
Consider the following homogeneous linear inequality system
\begin{equation}\label{eq:active-cone}
    -A_B^\top v\leq 0,Au=0, u_{N\cup B_2}\geq 0, \frac 1R(c^\top u+b^\top v)\leq 0 \ ,
\end{equation}
and denote $\mathcal K:=\left\{(u,v)\;|\;-A_B^\top v\leq 0,Au=0, u_{N\cup B_2}\geq 0, \tfrac{1}{R}(c^\top u+b^\top v)\leq 0\right\}$ the feasible set of \eqref{eq:active-cone}.
Denote $\alpha_{L_1}$ the sharpness constant to \eqref{eq:active-cone}, i.e., for any $(u,v)\in\mathbb R^{m+n}$,
\begin{equation}\label{eq:def_alphaL}
    \alpha_{L_1}\mathrm{dist}_2((u,v),\mathcal K)\leq \left\|\begin{pmatrix}
        [-A_B^\top v]^+ \\ Au \\ [-u_{N\cup B_2}]^+ \\ \tfrac{1}{R}[c^\top u+b^\top v]^+ \end{pmatrix}\right\|_2   \ .
\end{equation}
We can view sharpness $\alpha_{L_1}$ as a local sharpness condition of a homogeneous linear inequality system around the converging optimal solution $z^*$. Following the results in \cite{pena2024easily,lu2023geometry}, $\alpha_{L_1}$ is an extension of the minimal non-zero singular value of a certain matrix, while the global sharpness constant, i.e., the Hoffman constant, can be viewed of an extension of the minimal non-zero singular value of all sub-matrices of the matrix. Notice that there can be exponentially many sub-matrices, thus $\alpha_{L_1}$ provides a much sharper bound than the global sharpness constant, which is overly conservative~\cite{pena2024easily,lu2023geometry}.

 Denote set $\mathcal Z_L^*:=\{(x, y)\;|\;-A_B^Ty\leq c_B,\; Ax=b,\; x_{N\cup B_2}\geq 0,\; \frac 1R(c^Tx+b^Ty)\leq 0\}=z^*+\mathcal K$. The following proposition shows that $\alpha_{L_1}$ provides a bound on the sharpness constant to $\mathcal Z_L^*$.
\begin{prop}\label{prop:sharp-stage-1}
Denote $\alpha_{\eta}^{L_1}=\frac{\eta\alpha_{L_1}}{\eta\alpha_{L_1}+2}$.
Then, for any outer iteration $n$ and inner iteration $k$ such that $\|z^{n,k}\|_2\leq R$, it holds that
\begin{equation*}
    \alpha_{\eta}^{L_1}\mathrm{dist}(z^{n,k},\mathcal Z_L^*) \leq \|z^{n,k}-\PDHG(z^{n,k})\| \ .
\end{equation*}
\end{prop}
\begin{proof}
    By~\cite[part (c) of Lemma 3]{lu2023geometry}, we know that for any $R>0$ and any $z$ with $\|z\|_2\leq R$,
    \begin{equation}\label{eq:eucildean-sharp-stage-1}
        \alpha_{L_1}\mathrm{dist}_2(z,\mathcal Z_L^*)\leq \mathrm{dist}_2(0,\mathcal F(z)) \ ,
    \end{equation}
    where $\mathcal F(z)=\mathcal F(x,y)=\begin{pmatrix}
        c+A^\top y+\partial \iota_{\mathbb R_+^n}(x) \\ b-Ax
    \end{pmatrix}$ is the sub-differential to \eqref{eq:minmax}. 

   Note that $\|T(z^{n,k})\|_2\leq R$ since $\|z^{n,k}\|_2\leq R$ and we have
   {\small
    \begin{equation*}
        \begin{aligned}
            \mathrm{dist}(z^{n,k},\mathcal Z^*) &\leq \mathrm{dist}(T(z^{n,k}),\mathcal Z^*) + \|z^{n,k}-T(z^{n,k})\|\leq \sqrt{\frac 2\eta} \mathrm{dist}_2(T(z^{n,k}),\mathcal Z^*)+ \|z^{n,k}-T(z^{n,k})\|\\
            &\leq \sqrt{\frac2\eta}\frac{1}{\alpha_{L_1}}\mathrm{dist}_2(0,\mathcal F(T(z^{n,k})))+ \|z^{n,k}-T(z^{n,k})\| \leq \frac{2}{\eta\alpha_{L_1}}\mathrm{dist}_{P_\eta^{-1}}(0,\mathcal F(T(z^{n,k})))+ \|z^{n,k}-T(z^{n,k})\|\\
            & \leq \frac{2}{\eta\alpha_{L_1}}\|z^{n,k}-T(z^{n,k})\|+ \|z^{n,k}-T(z^{n,k})\| = \pran{1+\frac{2}{\eta\alpha_{L_1}}}\|z^{n,k}-T(z^{n,k})\|= \frac{1}{\alpha^{L_1}_\eta}\|z^{n,k}-T(z^{n,k})\| \ ,
        \end{aligned}
    \end{equation*}
    }
    where the third inequality uses \eqref{eq:eucildean-sharp-stage-1} at $T(z^{n,k})$ and the last inequality utilizes  $P_\eta(z^{n,k}-T(z^{n,k}))\in \mathcal F(T(z^{n,k}))$ by update rule of PDHG \eqref{eq:pdhg}.
\end{proof}

The next theorem shows that rHPDHG can identify the set $N$ and $B_1$ after a finite number of PDHG iterations.

\begin{thm}[Finite time identification]\label{thm:stage-1}
Consider restarted Halpern PDHG (Algorithm \ref{alg:hpdhg-restart}) for solving \eqref{eq:minmax} with step-size $\eta\leq \frac{1}{2\|A\|_2}$. Let $\{z^{n,k}=(x^{n,k},y^{n,k})\}$ be the iterates of the algorithm and let $z^*$ be the converging optimal solution, i.e., $z^{n,k} \rightarrow z^*$ as $n\rightarrow \infty$. Denote $(N, B_1, B_2)$ the partition of the primal variable indices given by Definition \ref{def:partition}. Denote $\alpha^{L_1}_{\eta}$ the sharpness constant in Proposition \ref{prop:sharp-stage-1}. Then, it holds that
    \begin{enumerate}
        \item[(i)] For outer iteration $n$ such that $\|z^{n,0}-\PDHG(z^{n,0})\|\geq \frac{\alpha_{\eta}^{L_1}\delta}{18\sqrt{2\eta}}$, the restart length $\tau^n$ can be upper bounded as
        \begin{equation*}
            \tau^n\leq \frac{72eR}{\alpha_{\eta}^{L_1}\delta} \ .
        \end{equation*}
        Furthermore, we know that $\max\left\{n: \|z^{n,0}-\PDHG(z^{n,0})\|\geq \frac{\alpha_{\eta}^{L_1}\delta}{18\sqrt{2\eta}} \right\}\leq \log\pran{\frac{72R}{\alpha_{\eta}^{L_1}\delta}}$.
        \item[(ii)] For outer iteration $n$ such that  $\|z^{n,0}-\PDHG(z^{n,0})\|<\frac{\alpha_{\eta}^{L_1}\delta}{18\sqrt{2\eta}}$, it holds for any $k\leq \tau^n$ that
        \begin{equation*}
            \|z^{n,k}-z^*\|_2\leq \frac{\delta}{2}, \ x^{n,k}_{B_1}>0, \ c^{n,k}_N-A_N^\top y^{n,k}>0 \ .
        \end{equation*}
    \end{enumerate}
\end{thm}

\begin{rem}
    Theorem \ref{thm:stage-1} implies that the identification complexity of Algorithm \ref{alg:hpdhg-restart} with step-size $\eta\leq \frac{1}{2\|A\|_2}$ equals $\widetilde O\pran{\frac{\|A\|_2}{\alpha_{L_1}\delta}}$. In contrast, vanilla PDHG requires $\widetilde O\pran{\frac{\|A\|_2^2}{(\alpha_{L_1}\delta)^2}}$ iterations for identification, this showcases that restarted Halpern iteration can achieve faster identification than PDHG.
\end{rem}

In the rest of this section, we will prove Theorem \ref{thm:stage-1}. We set a constant $C= 18$. The value of $C$ is obtained by the method of undetermined coefficients. We utilize $C$ instead of $18$ herein because this can make the proof easier to understand.

We begin with a collection of basic properties, which are summarized in the next lemma.
\begin{lem}\label{lem:stage-1}
    Suppose $\|z^{n,0}-\PDHG(z^{n,0})\|< \alpha^{L_1}_{\eta}\delta/(\sqrt{2\eta}C)$. Then it holds that
    \begin{enumerate}
        \item[(i)] $\mathrm{dist}_2(z^{n,0},\mathcal Z_L^*)< \frac{\delta}{C}$
        \item[(ii)] $\|z^{n,k}-z^{n,0}\|\leq \frac{k}{2}\frac{\alpha^{L_1}_{\eta}\delta}{\sqrt{2\eta}C}$ for any $k\geq 0$
        \item[(iii)] $\|z^{n,k}-\PDHG(z^{n,k})\|\leq \frac{\alpha^{L_1}_{\eta}\delta}{\sqrt{2\eta}C}$ and thus $\mathrm{dist}_2(z^{n,k},\mathcal Z_L^*)< \frac{\delta}{C}$ for any $k\geq 0$
        \item[(iv)] $\|z^{n,k+1}-z^{n,k}\|\leq \frac{3}{2}\frac{\alpha^{L_1}_{\eta}\delta}{\sqrt{2\eta}C}$ and thus $\|z^{n,k+1}-z^{n,k}\|_2\leq \frac{3}{2}\frac{\alpha^{L_1}_{\eta}\delta}{C}$
    \end{enumerate}
\end{lem}
\begin{proof}
    (i) Note that 
    \begin{equation*}
    \begin{aligned}
        \mathrm{dist}_2(z^{n,0},\mathcal Z_L^*) & \leq \sqrt{2\eta}\mathrm{dist}(z^{n,0},\mathcal Z_L^*) \leq \frac{\sqrt {2\eta}}{\alpha^{L_1}_{\eta}}\|z^{n,0}-\PDHG(z^{n,0})\|\leq \frac{\delta}{C}
    \end{aligned}
    \end{equation*}
    where the first inequality is from Lemma \ref{lem:change-of-norm} and the second inequality follows Proposition \ref{prop:sharp-stage-1}.

    (ii) We prove by induction. Suppose for some $k$, $\|z^{n,k-1}-z^{n,0}\|\leq \frac{k-1}{2}\frac{\alpha^{L_1}_{\eta}\delta}{\sqrt{2\eta}C}$. Then it holds that
    \begin{equation*}
        \begin{aligned}
            \|z^{n,k}-z^{n,0}\| & =\left\| \frac{k}{k+1}\PDHG(z^{n,k-1})+\frac{1}{k+1}z^{n,0}-z^{n,0} \right\|= \frac{k}{k+1}\left\|\PDHG(z^{n,k-1})-z^{n,0} \right\|\\
            & \leq \frac{k}{k+1}\left\|\PDHG(z^{n,k-1})-\PDHG(z^{n,0}) \right\| + \frac{k}{k+1}\left\|\PDHG(z^{n,0})-z^{n,0} \right\|\\
            & \leq \frac{k}{k+1}\left\|z^{n,k-1}-z^{n,0} \right\| + \frac{k}{k+1}\left\|\PDHG(z^{n,0})-z^{n,0} \right\|\\
            & \leq \frac{k}{k+1}\frac{k-1}{2}\frac{\alpha^{L_1}_{\eta}\delta}{\sqrt{2\eta}C} + \frac{k}{k+1}\frac{\alpha^{L_1}_{\eta}\delta}{\sqrt{2\eta}C}=\frac{k}{2}\frac{\alpha^{L_1}_{\eta}\delta}{\sqrt{2\eta}C}
        \end{aligned}
    \end{equation*}
    where the first equality is from the update rule of Halpern PDHG \eqref{eq:hpdhg}, the first inequality utilizes triangle inequality and the second and third inequalities uses non-expansiveness of PDHG operator $T$ and induction assumption, respectively. Thus we conclude the proof of this part by induction.
    
    (iii) Due to Lemma \ref{lem:sublinear-feas}, we have
    \begin{equation*}
        \|z^{n,k}-\PDHG(z^{n,k})\|\leq \frac{2}{k}\left\|z^{n,k}-z^{n,0}\right\|\leq \frac{\alpha^{L_1}_{\eta}\delta}{\sqrt{2\eta}C} \ .
    \end{equation*}
    Hence we know $\mathrm{dist}_2(z^{n,k},\mathcal Z_L^*)< \frac{\delta}{C}$ by part (i).
    
    (iv) It holds that
    \begin{equation*}
        \begin{aligned}
            \|z^{n,k+1}-z^{n,k}\| & = \left\| \frac{k+1}{k+2}\PDHG(z^{n,k})+\frac{1}{k+2}z^{n,0}-z^{n,k} \right\|\\
            & \leq \frac{k+1}{k+2}\left\| \PDHG(z^{n,k})-z^{n,k}\right\|+\frac{1}{k+2}\left\|z^{n,0}-z^{n,k} \right\|\\
            & \leq \frac{k+1}{k+2}\frac{\alpha^{L_1}_{\eta}\delta}{\sqrt{2\eta}C}+\frac{1}{k+2}\frac{k}{2}\frac{\alpha^{L_1}_{\eta}\delta}{\sqrt{2\eta}C}=\frac{3k+2}{2k+4}\frac{\alpha^{L_1}_{\eta}\delta}{\sqrt{2\eta}C}\leq \frac{3}{2}\frac{\alpha^{L_1}_{\eta}\delta}{\sqrt{2\eta}C} \ ,
        \end{aligned}
    \end{equation*}
    where the second inequality uses part (iii) and (ii).    
    Furthermore, since $\|z\|_2\leq \sqrt{2\eta}\|z\|$ for any $z$ with step-size $\eta\leq \frac{1}{2\|A\|_2}$,
    \begin{equation*}
        \|z^{n,k+1}-z^{n,k}\|_2\leq \sqrt{2\eta}\|z^{n,k+1}-z^{n,k}\|\leq \frac{3}{2}\frac{\alpha^{L_1}_{\eta}\delta}{C} \ .
    \end{equation*}
\end{proof}

The next lemma shows that if the fixed point residual is small, then all the subsequent iterates will be close to the converging optimal solution.
\begin{lem}\label{lem:ball}
    Suppose $\|z^{n,0}-\PDHG(z^{n,0})\|< \alpha_\eta^{L_1}\delta/(\sqrt{2\eta}C)$ with $C= 18$. Then it holds for any $n'\geq n$ that
    \begin{equation*}
        \|z^{n',0}-z^*\|_2\leq \frac{\delta}{2} .
    \end{equation*}
\end{lem}
\begin{proof}
    First we prove the following claim. For any $\Delta>1$ and $C>14\Delta$, suppose $\|z^{n,0}-\PDHG(z^{n,0})\|< \alpha_{L_1}\delta/(\sqrt{2\eta}C)$ but $\|z^{n,k}-z^*\|_2>\frac{\delta}{2\Delta}$ for some $k\geq 0$. Then the following two properties hold:
    \begin{enumerate}
        \item[(i)] $\left\|P_{\mathcal Z_L^*}(z^{n,k})-z^*\right\|_2>\frac{\delta}{2\Delta}$.
        \item[(ii)] For any $k'\geq k$, $\|z^{n,k'}-z^*\|_2>\frac{\delta}{2\Delta}$.
    \end{enumerate}
    (i) We prove by contradiction. If otherwise, i.e., $\left\|P_{\mathcal Z_L^*}(z^{n,k})-z^*\right\|_2\leq \frac{\delta}{2\Delta}$, then $P_{\mathcal Z_L^*}(z^{n,k})\in \mathcal Z_L^*\cap B_{\frac{\delta}{2}}(z^*) \subseteq \mathcal Z^*$. Note that $C> 14\Delta > 8\Delta$ and we have
    \begin{equation*}
        \begin{aligned}
            \|z^{n,k}-z^*\|_2 \leq \sqrt{2\eta}\|z^{n,k}-z^*\| \leq 2\sqrt{2\eta}\mathrm{dist}(z^{n,k},\mathcal Z^*)= 2\sqrt{2\eta}\mathrm{dist}(z^{n,k},\mathcal Z_L^*)\leq 4\mathrm{dist}_2(z^{n,k},\mathcal Z_L^*)\leq \frac{4\delta}{C}<\frac{\delta}{2\Delta} \ ,
        \end{aligned}
    \end{equation*}
    where the first inequality uses Lemma \ref{lem:change-of-norm}, the second one is due to part (iv) of Lemma \ref{lem:iterates-feas} and the fourth inequality follows from part (iii) of Lemma \ref{lem:stage-1}. This leads to contradiction with condition $\|z^{n,k}-z^*\|_2>\frac{\delta}{2\Delta}$. Thus it holds that $\left\|P_{\mathcal Z_L^*}(z^{n,k})-z^*\right\|_2> \frac{\delta}{2\Delta}$.
    
    (ii) We prove by contradiction. Suppose there exists a $k_0\geq k$ such that $\|z^{n,k_0}-z^*\|_2>\frac{\delta}{2\Delta}$ and $\|z^{n,k_0+1}-z^*\|_2\leq \frac{\delta}{2\Delta}$. Then we know {$\left\|P_{\mathcal Z_L^*}(z^{n,k_0})-z^*\right\|_2>\frac{\delta}{2\Delta}$}, and $\left\|P_{\mathcal Z_L^*}(z^{n,k_0+1})-z^*\right\|_2\leq \frac{\delta}{2\Delta}$ which implies $P_{\mathcal Z_L^*}(z^{n,k_0+1})\in \mathcal Z^*$, where the first inequality follows from part (i) and the second inequality is the consequence of non-expansiveness of projection operator.

    By triangle inequality, it holds that
    \begin{equation}\label{eq:ball-1}
        \begin{aligned}
            \|z^{n,k_0}-z^{n,k_0+1}\|_2 & \geq \|z^{n,k_0}-P_{\mathcal Z_L^*}(z^{n,k_0+1})\|_2-\|P_{\mathcal Z_L^*}(z^{n,k_0+1})-z^{n,k_0+1}\|_2 \\
            & \geq \|z^{n,k_0}-P_{\mathcal Z_L^*}(z^{n,k_0+1})\|_2 - \frac{\delta}{C}\\
            & \geq \|z^{n,k_0}-z^*\|_2-\|z^*-P_{\mathcal Z_L^*}(z^{n,k_0+1})\|_2-\frac{\delta}{C}\\
            & \geq \|z^{n,k_0}-z^*\|_2-\frac{\delta}{2\Delta}-\frac{\delta}{C} \ ,
        \end{aligned}
    \end{equation}
    where the second inequality uses $\|P_{\mathcal Z_L^*}(z^{n,k_0+1})-z^{n,k_0+1}\|_2=\mathrm{dist}_2(z^{n,k_0+1},\mathcal Z_L^*)\leq \frac{\delta}{C}$ by part (iii) of Lemma \ref{lem:stage-1} while the last inequality follows from the choice of $k_0$ such that $\|P_{\mathcal Z_L^*}(z^{n,k_0+1})-z^*\|_2\leq \frac{\delta}{2\Delta}$.
    
    Let $\tilde z=z^*+\frac{\delta}{2\Delta\|P_{\mathcal Z_L^*}(z^{n,k_0})-z^*\|_2}\pran{P_{\mathcal Z_L^*}(z^{n,k_0})-z^*}$ and then $\|\tilde z-z^*\|_2=\frac{\delta}{2\Delta}$ which implies $\tilde z\in\mathcal Z^*$. Note that
    \begin{equation}\label{eq:ball-2}
        \begin{aligned}
            \|z^{n,k_0}-z^*\|_2^2 &\geq \left\|z^{n,k_0}-P_{\mathcal Z_L^*}(z^{n,k_0})\right\|_2^2+\left\|P_{\mathcal Z_L^*}(z^{n,k_0})-z^*\right\|_2^2\geq \left\|P_{\mathcal Z_L^*}(z^{n,k_0})-z^*\right\|_2^2\\
            &=\pran{\left\|P_{\mathcal Z_L^*}(z^{n,k_0})-\tilde z\right\|_2+\left\|\tilde z-z^*\right\|_2}^2=\pran{\left\|P_{\mathcal Z_L^*}(z^{n,k_0})-\tilde z\right\|_2+\frac{\delta}{2\Delta}}^2\\
            & \geq \pran{\left\|z^{n,k_0}-\tilde z\right\|_2-\left\|z^{n,k_0}-P_{\mathcal Z_L^*}(z^{n,k_0})\right\|_2+\frac{\delta}{2\Delta}}^2\geq \pran{\frac{\delta}{4\Delta}-\frac{\delta}{C}+\frac{\delta}{2\Delta}}^2 \ ,
        \end{aligned}
    \end{equation}
    where the first inequality utilizes the fact that $P_{\mathcal Z_L^*}(z^{n,k_0})$ is a projection onto a convex set $\mathcal Z_L^*$, the first equality is due to the definition of $\tilde z$, and the last inequality follows from $|P_{\mathcal Z_L^*}(z^{n,k_0})-z^{n,k_0}\|_2=\mathrm{dist}_2(z^{n,k_0},\mathcal Z_L^*)\leq \frac{\delta}{C}$ and $\left\|z^{n,k_0}-\tilde z\right\|_2 \geq \frac{1}{\sqrt{2\eta}}\left\|z^{n,k_0}-\tilde z\right\| \geq \frac{1}{\sqrt{2\eta}}\left\|z^*-\tilde z\right\|\geq \frac{1}{2}\|z^*-\tilde z\|_2=\frac{\delta}{4\Delta}$ by part (iv) of Lemma \ref{lem:iterates-feas}.

    Combine \eqref{eq:ball-1} and \eqref{eq:ball-2}, and by noticing $C>14\Delta$ and $\alpha^{L_1}_\eta\leq1$ we achieve at
    \begin{equation*}
        \|z^{n,k_0}-z^{n,k_0+1}\|_2 \geq \|z^{n,k_0}-z^*\|_2-\frac{\delta}{2\Delta}-\frac{\delta}{C} \geq \pran{\frac{\delta}{4\Delta}-\frac{\delta}{C}+\frac{\delta}{2\Delta}}-\frac{\delta}{2\Delta}-\frac{\delta}{C}=\frac{\delta}{4\Delta}-\frac{2\delta}{C}>\frac{3\delta}{2C}\geq \frac{3\alpha^{L_1}_{\eta}\delta}{2C}\ .
    \end{equation*}
    This contradicts with part (iv) of Lemma \ref{lem:stage-1} and leads to $\|z^{n,k'}-z^*\|_2>\frac{\delta}{2\Delta}$ for any $k'\geq k$.
    
    Now we prove the lemma by induction and contradiction. Suppose $\|z^{n,0}-z^*\|_2>\frac{\delta}{2}$. Denote $\Delta_t=\frac{1}{1-\frac{2}{eC}-\cdots-\frac{2}{e^tC}}$ and $\widetilde \Delta = \frac{1}{1-\frac{2}{C}\frac{1/e}{1-1/e}}>\frac{1}{1-\sum_{i=0}^t\frac{2}{Ce^i}}=\Delta_t$ for any $t\geq 0$. Note that $C>14\widetilde\Delta>14\widetilde\Delta_t$ for $C\geq 18$. First by \eqref{eq:adaptive-restart} we have 
    \begin{equation*}
        \|T(z^{n+t,\tau^{n+t}})-z^{{n+t},\tau^{n+t}}\|_2\leq \sqrt{2\eta}\|T(z^{n+t,\tau^{n+t}})-z^{n+t,\tau^{n+t}}\|\leq \sqrt{2\eta}\frac
    1{e^{t+1}}\|T(z^{n,0})-z^{n,0}\|\leq \frac{\alpha_\eta^{L_1}\delta}{e^{t+1}C}<\frac{\delta}{e^{t+1}C} \ ,
    \end{equation*}
    where the first inequality is due to Lemma \ref{lem:change-of-norm} and the third inequality follows from the condition $\|z^{n,0}-\PDHG(z^{n,0})\|< \alpha_\eta^{L_1}\delta/(\sqrt{2\eta}C)$. Hence,
    \begin{equation*}
    \begin{aligned}
        \|z^{n+1,0}-z^*\|_2=\|T(z^{n,\tau^n})-z^*\|_2\geq \|z^{n,\tau^n}-z^*\|_2-\|T(z^{n,\tau^n})-z^{n,\tau^n}\|_2 > \frac{\delta}{2}-\frac{\delta}{eC}=\frac{\delta}{2\Delta_1} \ .
    \end{aligned}    
    \end{equation*}
    Thus by part (ii) we know for any $k\geq 0$, $\|z^{n+1,k}-z^*\|_2\geq \frac{\delta}{2\Delta_1}$. Suppose for some $t\geq 0$, $\|z^{n+t,0}-z^*\|_2\geq \frac{\delta}{2\Delta_t}$. Since $C>14\widetilde\Delta>14\widetilde\Delta_{t}$, by part (ii) we have that for any $k\geq 0$, $\|z^{n+t,k}-z^*\|_2\geq \frac{\delta}{2\Delta_{t}}$. Then
    {\small
    \begin{equation*}
    \begin{aligned}
        \|z^{n+t+1,0}-z^*\|_2=\|T(z^{n+t,\tau^{n+t}})-z^*\|_2\geq \|z^{n+t,\tau^{n+t}}-z^*\|_2-\|T(z^{n+t,\tau^{n+t}})-z^{n+t,\tau^{n+t}}\|_2 > \frac{\delta}{2\Delta_t}-\frac{\delta}{e^{t+1}C}=\frac{\delta}{2\Delta_{t+1}} \ .
    \end{aligned}    
    \end{equation*}
    }
    Thus it holds for any $t\geq 0$ that 
    {
    \begin{equation*}
    \begin{aligned}
        \|z^{n+t,0}-z^*\|_2\geq \frac{\delta}{2\Delta_{t}}>\frac{\delta}{2\widetilde\Delta}>0 \ ,
    \end{aligned}    
    \end{equation*}
    }
    which contradicts $z^*=\lim_{t\rightarrow \infty}z^{n+t,0}$ and thus we can conclude that $\|z^{n,0}-z^*\|_2\leq\frac{\delta}{2}$.
\end{proof}

Now, we are ready to prove Theorem \ref{thm:stage-1}.

\begin{proof}[Proof of Theorem \ref{thm:stage-1}]
    Recall the constant $C= 18$.

    (i) Since $\|z^{n,0}-\PDHG(z^{n,0})\|\geq \alpha^{L_1}_{\eta}\delta/(\sqrt{2\eta}C)$ and $\mathrm{dist}(z^{n,0},\mathcal Z^*)\leq \sqrt{\frac 2\eta}R$ by Lemma \ref{lem:change-of-norm}, we have
        \begin{equation*}
            \alpha_n:=\frac{\|z^{n,0}-\PDHG(z^{n,0})\|}{\mathrm{dist}(z^{n,0},\mathcal Z^*)} \geq \frac{\alpha^{L_1}_{\eta}\delta}{2CR} \ .
        \end{equation*}
        Thus, the restart frequency
        \begin{equation*}
            \tau^n\leq \frac{2e}{\alpha_n}\leq 
            \frac{4eCR}{\alpha^{L_1}_{\eta}\delta}=\frac{72eR}{\alpha^{L_1}_{\eta}\delta}
        \end{equation*}
    Furthermore, note that $\|z^{n,0}-\PDHG(z^{n,0})\|\leq \pran{\frac{1}{e}}^n\|z^{0,0}-\PDHG(z^{0,0})\|\leq \pran{\frac{1}{e}}^n 2\mathrm{dist}(z^{0,0}\mathcal Z^*)\leq \frac{2\sqrt{\frac2\eta}R}{e^n}$ where the last inequality uses Lemma \ref{lem:change-of-norm}, and thus 
    \begin{equation*}
        \max\left\{n: \|z^{n,0}-\PDHG(z^{n,0})\|\geq \frac{\alpha_{\eta}^{L_1}\delta}{18\sqrt{2\eta}} \right\}=\max\left\{n: \|z^{n,0}-\PDHG(z^{n,0})\|\geq \frac{\alpha_{L_1}\delta}{\sqrt{2\eta}C} \right\}\leq \log\pran{\frac{4RC}{\alpha^{L_1}_{\eta}\delta}}=\log\pran{\frac{72R}{\alpha^{L_1}_{\eta}\delta}} \ .
    \end{equation*}

    (ii) From Lemma \ref{lem:ball}, we know that
    $\|z^{n,0}-z^*\|_2\leq \frac{\delta}{2}$ and for any $k\leq \tau^n$, $\|z^{n,k}-z^*\|_2\leq \frac{\delta}{2}$. The identification $x^{n,k}_{B_1}>0, \ c^{n,k}_N+A_N^\top y^{n,k}>0$ is thus guaranteed from \cite[Lemma 4]{lu2023geometry}.
\end{proof}

\subsection{Stage II: local linear convergence}
In previous section, it has been demonstrated that after a specific number of iterations, the iterates generated by Algorithm \ref{alg:hpdhg-restart} maintain proximity to an optimal solution, with the ability to identify the non-degenerate coordinates set denoted as $N$ and $B_1$. This section aims to investigate the local dynamics after identification. The local linear convergence rate is characterized through a local sharpness constant denoted as $\alpha_{L_2}$ with respect to a homogeneous linear inequality system. Notably, the local convergence rate is much faster than the global linear convergence rate as derived in~\cite{applegate2023faster}, where the latter is characterized by the global Hoffman constant associated with the KKT system of the LP.

Denote $R_2=\delta$. We consider the following homogeneous linear inequality system 
\begin{equation}\label{eq:local-cone}
    -A_B^\top v\leq 0,Au=0, u_{N\cup B_2}\geq 0, \frac 1{R_2}(c^\top u+b^\top v)\leq 0 \ ,
\end{equation}
and denote $\mathcal K:=\left\{(u,v)\;|\;-A_B^\top v\leq 0,Au=0, u_{N\cup B_2}\geq 0, \tfrac{1}{R_2}(c^\top u+b^\top v)\leq 0\right\}$.
Denote $\alpha_{L_2}$ the sharpness constant to \eqref{eq:local-cone}, i.e., for any $(u,v)\in\mathbb R^{n+m}$,
\begin{equation}\label{eq:def_alphaL2}
    \alpha_{L_2}\mathrm{dist}_2((u,v),\mathcal K)\leq \left\|\begin{pmatrix}
        [-A_B^\top v]^+ \\ Au \\ [-u_{N\cup B_2}]^+ \\ \tfrac{1}{R_2}[c^\top u+b^\top v]^+ \end{pmatrix}\right\|_2   \ .
\end{equation}

Following the same proof as Proposition \ref{prop:sharp-stage-1}, the following property holds.
\begin{prop}\label{prop:sharp-stage-2}
Denote $\alpha^{L_2}_\eta = \frac{s\alpha_{L_2}}{s\alpha_{L_2}+2}$. Then for any outer iteration $n$ and inner iteration $k$ such that $z^{n,k}\in B_{\delta}(z^*)$, it holds that
\begin{equation*}
    \alpha^{L_2}_\eta\mathrm{dist}(z^{n,k},\mathcal Z^*) \leq \|z^{n,k}-\PDHG(z^{n,k})\| \ .
\end{equation*}
\end{prop}

Theorem \ref{thm:stage-2} presents the local linear convergence of rHPDHG after identification. The driving force of this stage is the local sharpness $\alpha^{L_2}_\eta$.
\begin{thm}\label{thm:stage-2}
Consider restarted Halpern PDHG (Algorithm \ref{alg:hpdhg-restart}) for solving \eqref{eq:minmax} with step-size $\eta\leq \frac{1}{2\|A\|_2}$. Let $\{z^{n,k}\}$ be the iterates of the algorithm and let $z^*$ be the converging optimal solution, i.e., $z^{n,k} \rightarrow z^*$ as $n\rightarrow \infty$. Denote $\alpha_{\eta}^{L_2}$ the sharpness constant in Proposition \ref{prop:sharp-stage-2}. Denote $N:=\max\left\{n: \|z^{n,0}-\PDHG(z^{n,0})\|\geq \frac{\alpha^{L_1}_{\eta}\delta}{18\sqrt{2\eta}} \right\}$. Then it holds for any $n>N$ that
\begin{enumerate}
    \item[(i)] The restart length $\tau^n$ can be upper bounded as
    \begin{equation*}
        \tau^n \leq \frac{2e}{\alpha_\eta^{L_2}} \ .
    \end{equation*}
    \item[(ii)] The distance to the converging optimal solution decays linearly,
    \begin{equation*}
        \|z^{n,0}-z^*\|\leq \frac{2}{\alpha_\eta^{L_2}}\pran{\frac 1e}^{n-N-1}\|z^{N+1,0}-\PDHG(z^{N+1,0})\| \ .
    \end{equation*}
\end{enumerate}   
\end{thm}

\begin{proof}
Since $n\geq N+1$, we know by part (ii) of Theorem \ref{thm:stage-1} that $\|z^{n,k}-z^*\|_2\leq \frac{\delta}{2}$ and thus $\alpha_n\geq \alpha_\eta^{L_2}$ by Proposition \ref{prop:sharp-stage-2}.
The proof of (i) follows directly from $\tau^n\leq \frac{2e}{\alpha_n}\leq \frac{2e}{\alpha_\eta^{L_2}}$. The proof of (ii) can be derived as follows
\begin{equation*}
    \begin{aligned}
        \|z^{n,0}-z^*\| \leq 2\mathrm{dist}(z^{n,0},\mathcal Z^*)\leq \frac{2}{\alpha_\eta^{L_2}}\|z^{n,0}-\PDHG(z^{n,0})\|\leq \frac{2}{\alpha_\eta^{L_2}}\pran{\frac 1e}^{n-N-1}\|z^{N+1,0}-\PDHG(z^{N+1,0})\| \ ,
    \end{aligned}
\end{equation*}
where the first inequality leverages part (iv) of Lemma \ref{lem:iterates-feas} and the second one follows from Proposition \ref{prop:sharp-stage-2}. The last inequality is due to the restart condition \eqref{eq:adaptive-restart}.
\end{proof}

\begin{rem}
Theorem \ref{thm:stage-2} implies the $\widetilde O\pran{\frac{\|A\|}{\alpha_{L_2}}\log\frac{1}{\epsilon}}$ local linear rate of restarted Halpern PDHG. In contrast, \cite{lu2023geometry} shows an $\widetilde O\pran{\pran{\frac{\|A\|}{\alpha_{L_2}}}^2\log\frac{1}{\epsilon}}$ local linear convergence rate for vanilla PDHG. This demonstrates the accelerated local linear rate of restarted Halpern PDHG.
\end{rem}

\section{Accelerated infeasibility detection of rHPDHG}\label{sec:infeas}
In this section, we discuss the behavior of restarted Halpern PDHG on infeasible LP. We show that rHPDHG can achieve linear convergence for infeasibility detection of LP without any additional assumptions. This is in contrast to the previous work on analyzing PDHG for LP's infeasibility detection~\cite{applegate2024infeasibility}, where the linear convergence rate holds only with a non-degenerate assumption. Furthermore, our linear convergence rate is faster than that in~\cite{applegate2024infeasibility} in the sense of removing a square term; thus, it is an accelerated linear convergence rate.

The studies of general operator splitting methods on infeasible problems are often based on a geometric object, the infimal displacement vector $v$:
\begin{prop}[{\cite[Lemma 4]{pazy1971asymptotic}}]
    Let $T$ be a nonexpansive operator, and then the set $\mathrm{cl}(\mathrm{Range}(\PDHG-I))$ is convex. Consequently, there exists a unique minimum norm vector $v$ in this set:
    \begin{equation*}
        v:=\argmin_{z\in\mathrm{cl}(\mathrm{Range}(\PDHG-I))}\|z\| \ ,
    \end{equation*}
    which is called the infimal displacement vector with respect to the operator $T$.
\end{prop}
Intuitively, the vector $v$ is the minimum size perturbation to ensure the perturbed operator has a fixed point. Suppose $T$ is the operator induced by PDHG \eqref{eq:pdhg} for solving infeasible \eqref{eq:lp}. In the context of LP, it is proven that the set $\mathrm{Range}(\PDHG-I)$ is closed~\cite[Proposition 3]{applegate2024infeasibility} and the vector $v$ is indeed a certificate of infeasibility of \eqref{eq:lp}, i.e., $v$ satisfies Farkas lemma.
\begin{prop}[{\cite[Theorem 4]{applegate2024infeasibility}}]\label{thm:full-infeas}
    Let $T$ be the operator induced by PDHG on \eqref{eq:lp} with step-size $\eta< \frac{1}{\|A\|}$, and let $\{z^k=(x^k,y^k)\}_{k=0}^{\infty}$ be a sequence generated by the fixed point iteration from an arbitrary starting point $z^0$. Then, one of the following holds:

    (a). If both primal and dual are feasible, then the iterates $(x^k,y^k)$ converge to a primal-dual solution $z^*=(x^*,y^*)$ and $v=(T-I)(z^*)=0$.

    (b). If both primal and dual are infeasible, then both primal and dual iterates diverge to infinity. Moreover, the primal and dual components of the infimal displacement vector $v=(v_x,v_y)$ give certificates of dual and primal infeasibility, respectively.

    (c). If the primal is infeasible and the dual is feasible, then the dual iterates diverge to infinity, while the primal iterates converge to a vector $x^*$. The dual-component $v_y$ is a certificate of primal infeasibility. Furthermore, there exists a vector $y^*$ such that $v=(T-I)(x^*,y^*)$.

    (d). If the primal is feasible and the dual is infeasible, then the same conclusions as in the previous item hold by swapping primal with dual.
\end{prop}

It turns out that the iterates of Halpern PDHG also encode the information of infeasibility and the certificate $v$ can be recovered by the normalized iterates $\frac{2z^{k}}{k}$ and the difference $T(z^k)-z^k$ with sublinear rate. The results are summarized in the following lemma.
\begin{lem}[{\cite[Theorem 7 and 8]{park2023accelerated}}]\label{lem:sublinear-infeas}
    Consider $\{z^{k}\}$ the iterates of Halpern PDHG for solving infeasible \eqref{eq:lp}. Let $v\in\mathrm{Range}(\PDHG-I)$ the unique minimizer $v:=\argmin_{u\in{\mathrm{Range}(\PDHG-I)}}\|u\|^2$. Denote $z^*\in \mathcal Z_v^*:=\{z^*\mid \PDHG(z^*)-z^*-v=0\}$. Then it holds for any $k\geq 1$ that
    \begin{equation*}
        \left\| \frac{2}{k}(z^{k}-z^0)-v \right\|\leq \frac{4}{k}\mathrm{dist}(z^0,\mathcal Z_v^*)\ ,
    \end{equation*}
    and
    \begin{equation*}
        \left\|\PDHG(z^k)-z^k-v\right\|\leq \frac{\sqrt{\log k+5}+1}{k+1}\mathrm{dist}(z^0,\mathcal Z_v^*) \ .
    \end{equation*}
\end{lem}

However, the previous results are mostly for the sublinear rate of Halpern PDHG. In this section, we will show that rHPDHG is able to achieve linear convergence to a certificate of infeasibility for LP. 

The main result is built upon the sublinear convergence in Lemma \ref{lem:sublinear-infeas}, the closeness to degeneracy of infimal displacement vector $\delta_v$ (which is an extension of the closeness to degeneracy $\delta$ for feasible problem defined in Section \ref{sec:two_stage}), an auxiliary problem \eqref{eq:auxiliary}, an auxiliary PDHG operator $\tilde T$, and its sharpness constant $\alpha_\eta$. 

More specifically, we first define the partition of coordinates:
\begin{equation*}
    \begin{aligned}
        & B=\{i\in[n]: (v_x)_i>0\}\\
        & N_1=\{i\in[n]: (v_x)_i=0,(A^\top v_y)_i=0\}\\
        & N_2=\{i\in[n]: (v_x)_i=0,(A^\top v_y)_i>0\} \ .
    \end{aligned}
\end{equation*}
With some abuse of notations, these three sets are defined based on the infimal displacement vector $v$ for infeasible/unbounded problems, and they are different from those three sets defined in Section \ref{sec:two_stage} based on the converging optimal solution $z^*$ for feasible and bounded problems. We further denote the closeness to degeneracy for infeasible problems as     
\begin{equation*}
        \delta_v = \min\left\{ \min_{i\in B}(v_x)_i, \ \min_{i\in N_2}\frac{A_i^\top v_y}{\|A\|_2}  \right\} \ .
    \end{equation*}

Using the three sets $B, N_1$ and $N_2$, we then define the following feasible and bounded auxiliary LP~\cite{applegate2024infeasibility}.
\begin{equation}\label{eq:lp-aux}
    \begin{aligned}
        \min_{x_B,x_{N_1},x_{N_2}}& \ \pran{c_B+\frac{(v_x)_B}{\eta}}^\top x_B+c_{N_1}^\top x_{N_1}+c_{N_2}^\top x_{N_2}\\
        \mathrm{s.t.}& \ A_Bx_B+A_{N_1}x_{N_1}+A_{N_2}x_{N_2}=b+\frac{v_y}{\eta}\\
        & \ x_{N_1}\geq 0,\ x_{N_2}\geq 0 \ .
    \end{aligned}
\end{equation}

The PDHG operator on solving \eqref{eq:lp-aux} is denoted as $\widetilde{T}(z)=z^+$ with the following update rule
\begin{equation}\label{eq:auxiliary}
\begin{cases}
        x_B^{+}\leftarrow x_B-\eta A_B^\top y-\eta c_B -(v_x)_B\\
        x_{N_1}^{+}\leftarrow \text{proj}_{\mathbb R^{|N_1|}_+}(x_{N_1}-\eta A_{N_1}^\top y-\eta c_{N_1})-(v_x)_{N_1}\\
        x_{N_2}^{+}\leftarrow \text{proj}_{\mathbb R^{|N_2|}_+}(x_{N_2}-\eta A_{N_2}^\top y-\eta c_{N_2})-(v_x)_{N_2}\\
        y^{+}\leftarrow y+\eta A(2x^{+}-x)-\eta b-v_y
    \end{cases} \ .
\end{equation}

In addition, let $\tilde z^*_0\in \mathrm{Fix}(\widetilde T)\cap\{z:x_{N_2}=0\}$ (whose existence is proven later in Lemma \ref{lem:auxiliary-lp}). We know by \cite{pena2021new} that there exists a constant $\alpha>0$ such that for any $z=(x,y)$ satisfying $\|z\|_2\leq 2(\|z^0-\tilde z^*_0\|_2+\|\tilde z^*_0\|_2)$ and $x_{N_2}=0$, it holds that
    \begin{equation*}
        \alpha\mathrm{dist}_2(z,\mathrm{Fix}(\widetilde T)\cap\{z:x_{N_2}=0\})\leq \mathrm{dist}_2(0,\mathcal F(z)) \ .
    \end{equation*}
Following the same proof of Proposition \ref{prop:sharp}, we know there exists a constant $\alpha_\eta=\frac{\eta\alpha}{\eta\alpha+2}$ such that for any $z=(x,y)$ satisfying $\|z\|_2\leq 2(\|z^0-\tilde z^*_0\|_2+\|\tilde z^*_0\|_2)$ and $x_{N_2}=0$,
\begin{equation*}
        \alpha_\eta\mathrm{dist}(z,\mathrm{Fix}(\widetilde T)\cap\{z:x_{N_2}=0\})\leq \|\widetilde T(z)-z\| \ .
    \end{equation*}

Theorem \ref{thm:infeas} presents the complexity of infeasibility detection of rHPDHG, which is the main result of this section.
\begin{thm}\label{thm:infeas}
    Consider $\{z^{n,0}\}$ the iterates of restarted Halpern-PDHG with fixed restart frequency. Then there exists an optimal solution $\tilde z_0^*\in \mathcal Z_v^*:=\{z^*\mid \PDHG(z^*)-z^*-v=0\}$ and $R=2(\|z^{0,0}-\tilde z_0^*\|_2+\|\tilde z^*_0\|_2)$ such that for
    \begin{equation*}
        k\geq k^*=\left\lceil\max\left\{\frac{12R}{C\delta_v}, \frac{(\sqrt 2+1)e}{\alpha_\eta}\log\pran{\frac{(\sqrt 2+1)e}{\alpha_\eta}}\right\}\right\rceil \ ,
    \end{equation*}
    where $C\leq \min\left\{\frac{1}{\sqrt 2(1+\eta\|A\|_2)},\frac{\eta}{\sqrt 2(1+\eta\|A\|_2)}\right\}$ with $\eta\leq \frac{1}{2\|A\|_2}$, it holds that
        \begin{equation*}
            \|T(z^{n,0})-z^{n,0}-v\| \leq 2e^{-n}\mathrm{dist}(z^{0,0},\mathcal Z_v^*) \ .
        \end{equation*}
\end{thm}
\begin{rem}
    Suppose step-size $\eta=\frac{D}{\|A\|_2}$ for some $D\leq \frac 12$. Then the complexity of restarted Halpern-PDHG with fixed frequency equals $\widetilde O\pran{\max\left\{\frac{\|A\|_2}{\delta_v}, \frac{\|A\|_2}{\alpha}\right\}\log\frac{1}{\epsilon}}$. We can achieve better complexity $\widetilde O\pran{\frac{\|A\|_2}{\delta_v}+\frac{\|A\|_2}{\alpha}\log\frac{1}{\epsilon}}$ if we restart the first epoch at $k\geq \frac{12R}{C\delta}$ and uses restart frequency $k\geq \frac{(\sqrt 2+1)e}{\alpha_\eta}\log\pran{\frac{(\sqrt 2+1)e}{\alpha_\eta}}$ afterwards. In comparison, the last iterates of vanilla PDHG converges to $\epsilon$ accuracy within $O\pran{\frac{\|A\|_2}{\delta_v}+\pran{\frac{\|A\|_2}{\alpha}}^2\log\frac{1}{\epsilon}}$ iterations.
\end{rem}

We present the proof of Theorem \ref{thm:infeas} in the rest of this section. The major challenge to establishing Theorem \ref{thm:infeas} is that the iterates of rHPDHG on solving infeasible LP go to infinity. Thus, the sharpness condition used in the analysis for feasible problems (i.e., Proposition \ref{prop:sharp}) does not hold. To overcome this difficulty,  we first derive a finite-time identification of Halpern iterates (Lemma \ref{lem:infeas-iden}). After identification, we connect the iterates of rHPDHG for the original LP with rHPDHG for the auxiliary LP \eqref{eq:auxiliary} (Lemma \ref{lem:auxiliary-lp}), {which is a feasible and bounded problem}. Utilizing such a connection, we develop a local sharpness condition of rHPDHG for the original LP (Lemma \ref{lem:infeas-sharp}), which leads to the conclusion of Theorem \ref{thm:infeas}.

First, we show Halpern PDHG can identify set $B$ and $N_2$ for infeasible/unbounded problems after a finite number of iterations.
\begin{lem}\label{lem:infeas-iden}
    Denote closeness to degeneracy $\delta = \min\left\{ \min_{i\in B}(v_x)_i, \min_{i\in N_2}\frac{A_i^\top v_y}{\|A\|}  \right\}$. Consider the iterates of Halpern PDHG $\{z^k\}$. Let $\tilde z_0^*\in \mathrm{Fix}(\widetilde T)\cap\{z:x_{N_2}=0\}$. Then for any $k\geq \frac{6\sqrt{2\eta}\|z^{0}-\tilde z_0^*\|}{C\delta_v}$ where $C\leq \frac{1}{\sqrt 2(1+\eta\|A\|_2)} \min \left\{1,\eta\right\}$ with $\eta\leq \frac{1}{2\|A\|_2}$, it holds that
    \begin{equation*}
        (T(z^k)_x)_B>0,\ (T(z^k)_x)_{N_2}=0 \ .
    \end{equation*}
\end{lem}
\begin{proof}
First we know from the choice of $k\geq \frac{6\sqrt{2\eta}\|z^{0}-\tilde z_0^*\|}{C\delta_v}$ and Lemma \ref{lem:sublinear-infeas} that
\begin{equation}\label{eq:iden}
    \left\|\frac{2}{k}(z^{k}-z^*)-v\right\|_2\leq\sqrt{2\eta}\left\|\frac{2}{k}(z^{k}-z^*)-v\right\|\leq \sqrt{2\eta}\frac{4}{k}\|z^0-z^*\|+\sqrt{2\eta}\frac{2}{k}\|z^0-z^*\|=\sqrt{2\eta}\frac{6}{k}\|z^0-z^*\|\leq C\delta_v \ ,
\end{equation}
where the first inequality uses $\|\cdot\|_2\leq \sqrt{2\eta} \|\cdot\|$ due to step-size $\eta\leq \frac{1}{2\|A\|_2}$ and the second inequality is from Lemma \ref{lem:sublinear-infeas}. In the following, we prove the identification of the non-basis coordinate $i\in B$ and the non-degenerate basis coordinate $i\in N_2$ separately, given the condition \eqref{eq:iden} holds.

Let $\tilde z_0^*=(x^*,y^*)\in \mathrm{Fix}(\widetilde T)\cap\{z:x_{N_2}=0\}$. Consider coordinate $i\in B$. 
{\small
    \begin{equation*}
        \begin{aligned}
            \left|\frac 2k A_i^\top (y^k-y^*)\right|=\left|\frac 2k A_i^\top (y^k-y^*) -A_i^\top v_y\right|\leq \left\|\frac 2k A^\top (y^k-y^*)-A^\top v_y\right\|_2\leq \|A\|_2 \left\|\frac 2k (y^k-y^*)-v_y\right\|_2\leq C\|A\|_2\delta_v \ ,
        \end{aligned}
    \end{equation*}
}
where the first equality uses $A_B^\top v_y=0$, and by \eqref{eq:iden},
    \begin{equation*}
        \left|\frac 2k (x_i^k-x_i^*) -(v_x)_i\right|\leq \left\|\frac 2k (x^k-x^*) -v_x\right\|_2\leq C\delta_v \ .
    \end{equation*}
Thus we have
    \begin{equation}\label{eq:b-1}
        A_i^\top y^k \leq A_i^\top y^*+\frac{k}{2}C\|A\|_2\delta_v \ ,
    \end{equation}
    and
    \begin{equation}\label{eq:b-2}
        x_i^k \geq x_i^*+\frac{k}{2}(v_x)_i-\frac{k}{2}C\delta_v \ .
    \end{equation}
    Therefore it holds that $(T(z^k)_x)_i>0$ for $i\in B$ by noticing that
    {\small
    \begin{equation*}
        \begin{aligned}
            x_i^k-\eta(c_i+A_i^\top y^k)\geq x_i^*+\frac{k}{2}(v_x)_i-\frac{k}{2}(1+\eta\|A\|_2)C\delta_v-\eta(c_i+A_i^\top y^*)\geq x_i^*+\frac{k}{2}(v_x)_i-\frac{k}{2}(1+\eta\|A\|_2)C\delta_v \geq x_i^*>0 \ ,
        \end{aligned}
    \end{equation*}
    }
    where the first inequality combines \eqref{eq:b-1} and \eqref{eq:b-2} and the second one follows from $c_i+A_i^\top y^*\leq 0$.

Now consider $i\in N_2$. By \eqref{eq:iden}, we have
    \begin{equation*}
        \begin{aligned}
            \left|\frac 2k A_i^\top (y^k-y^*) -A_i^\top v_y\right|\leq \left\|\frac 2k A^\top (y^k-y^*)-A^\top v_y\right\|_2\leq \|A\|_2 \left\|\frac 2k (y^k-y^*)-v_y\right\|_2\leq C\|A\|_2\delta_v
        \end{aligned}
    \end{equation*}
    and
    \begin{equation*}
        \left|\frac 2k (x_i^k-x_i^*)\right|=\left|\frac 2k (x_i^k-x_i^*) -(v_x)_i\right|\leq \left\|\frac 2k (x^k-x^*)-v_x\right\|_2\leq C\delta_v
    \end{equation*}
    where the first inequality utilizes $(v_x)_i=0$ for $i\in N_2$ and the last one uses \eqref{eq:iden}. Thus it holds that
    \begin{equation}\label{eq:n2-1}
        A_i^\top y^k \geq A_i^\top y^*+\frac{k}{2}A_i^\top v_y-\frac{k}{2}C\|A\|_2\delta_v
    \end{equation}
    and
    \begin{equation}\label{eq:n2-2}
        x_i^k \leq x_i^*+\frac{k}{2}C\delta_v=\frac{k}{2}C\delta_v\ .
    \end{equation}
    Hence we have $(T(z^k)_x)_i=0$ for $i\in N_2$ since
    \begin{equation*}
        \begin{aligned}
            x_i^k-\eta(c_i+A_i^\top y^k)\leq \frac{k}{2}(1+\eta\|A\|_2)C\delta_v-\eta(c_i+A_i^\top y^*)-\eta\frac{k}{2}A_i^\top v_y\leq-\eta(c_i+A_i^\top y^*)< 0 \ ,
        \end{aligned}
    \end{equation*}
    where the first inequality follows from \eqref{eq:n2-1} and \eqref{eq:n2-2}, and the second inequality utilizes the definition of $C$ and $\delta$. 
\end{proof}

Nest, we build up the relationship between operators $T$ for the original infeasible problem and the operator $\widetilde T$ for the feasible auxiliary problem after identification.
\begin{lem}\label{lem:auxiliary-lp}
    Let $v\in\mathrm{Range}(T-I)$ be the infimal displacement vector.
    \begin{enumerate}
        \item[(i)] For any $z$ such that $(T(z)_x)_B>0$ and $(T(z)_x)_{N_2}=0$, it holds that for any $\lambda\geq 0$
        \begin{equation*}
            \widetilde{T}(z-\lambda v)=T(z)-(\lambda+1)v \ .
        \end{equation*}
        \item[(ii)] There exists a fixed point $\tilde z^*\in\mathrm{Fix}(\widetilde T)$ such that $\tilde x^*_{N_2}=0$, i.e., $\mathrm{Fix}(\widetilde T)\cap\{z:x_{N_2}=0\}\neq \O$ \ .
        \item[(iii)] For any $\tilde z^*\in \mathrm{Fix}(\widetilde T)\cap\{z:x_{N_2}=0\}$, it holds that for any $\lambda\geq 0$ that
        \begin{equation*}
            T(\tilde z^*+\lambda v)-(\tilde z^*+\lambda v)=v \ ,
        \end{equation*}
        i.e., $\{\tilde z^*+\lambda v:\tilde z^*\in \mathrm{Fix}(\widetilde T)\cap\{z:x_{N_2}=0\},\lambda \geq 0\}\subseteq \mathcal Z_v^*$.
    \end{enumerate}
\end{lem}
\begin{proof}
    (i) Note that $A_B^\top v_y=0$ and $Av_x=0$, and we have
    \begin{equation*}
        \begin{aligned}
            & (\widetilde T(z-\lambda v)_x)_B =x_B-\lambda(v_x)_B-\eta\pran{c_B+A_B^\top (y-\lambda v_y)}-(v_x)_B=x_B-\eta\pran{c_B+A_B^\top y}-(\lambda+1)(v_x)_B \ .\\
            & \qquad\qquad\qquad \  =(T(z)_x)_B-(\lambda+1)(v_x)_B \ .\\
            & (\widetilde T(z-\lambda v)_x)_{N_1} =\text{proj}_{\mathbb R^{|N_1|}_+}(x_{N_1}-\lambda(v_x)_{N_1}-\eta A_{N_1}^\top y-\eta c_{N_1})-(v_x)_{N_1}=(T(z)_x)_{N_1}=(T(z)_x)_{N_1}-(\lambda+1)(v_x)_{N_1}\ .\\
            & (\widetilde T(z-\lambda v)_x)_{N_2} =\text{proj}_{\mathbb R^{|N_2|}_+}(x_{N_2}-\lambda(v_x)_{N_2}-\eta A_{N_2}^\top y-\eta c_{N_2})-(v_x)_{N_2}=0=(T(z)_x)_{N_2}-(\lambda+1)(v_x)_{N_2}\ .\\
            &\widetilde T(z-\lambda v)_y =y-\lambda v_y+\eta A(2T(z-\lambda v)_x-(x-\lambda v_x))-\eta b-v_y\\
            & \qquad\qquad\ \ = y+\eta A(2T(z)_x-x-(\lambda+2) v_x))-\eta b-(\lambda+1)v_y\\
            & \qquad\qquad\ \ =T(z)_y-(\lambda+1) v_y\ .
        \end{aligned}
    \end{equation*}
    
    (ii) Let $z^*$ satisfy $T(z^*)-z^*=v$. Let $K$ be the iteration such that for any $k\geq K$, $(T^k(z^*)_x)_B>0$ and $(T^k(z^*)_x)_{N_2}=0$ (the existence of such $K$ comes from \cite[Fact 1]{applegate2024infeasibility}) and thus $\widetilde{T}(T^K(z^*))=T^{K+1}(z^*)-v$ by part (i). Note that $T^{K+1}(z^*)-v=T^{K}(z^*)$ by \cite[Lemma 2]{applegate2024infeasibility} and we have $\widetilde{T}(T^K(z^*))=T^{K+1}(z^*)-v=T^{K}(z^*)$, i.e., $T^K(z^*)\in \mathrm{Fix}(\widetilde T)\cap\{z:x_{N_2}=0\}$.

    (iii) The proof follows from direct calculation as shown below: 
    \begin{equation*}
        \begin{aligned}
            & (T(\tilde z^*+\lambda v)_x)_B =\text{proj}_{\mathbb R^{|B|}_+}\pran{\tilde x_B^*+\lambda(v_x)_B-\eta\pran{c_B+A_B^\top (\tilde y^*+\lambda v_y)}}=\text{proj}_{\mathbb R^{|B|}_+}\pran{\tilde x_B^*+(\lambda+1)(v_x)_B}\\
            & \qquad\qquad\qquad \ \  =\tilde x_B^*+(\lambda+1)(v_x)_B\ .\\
            & (T(\tilde z^*+\lambda v)_x)_{N_1} =\text{proj}_{\mathbb R^{|N_1|}_+}(\tilde x_{N_1}^*+\lambda(v_x)_{N_1}-\eta A_{N_1}^\top (\tilde y^*+\lambda v_y)-\eta c_{N_1})\\
            & \qquad\qquad\qquad \ \ \ =\text{proj}_{\mathbb R^{|N_1|}_+}(\tilde x_{N_1}^*-\eta A_{N_1}^\top \tilde y^*-\eta c_{N_1})=\tilde x_{N_1}^*=\tilde x_{N_1}^*+(\lambda+1)(v_x)_{N_1}\ .\\
            & (T(\tilde z^*+\lambda v)_x)_{N_2} =\text{proj}_{\mathbb R^{|N_2|}_+}(\tilde x_{N_2}^*+\lambda (v_x)_{N_2}-\eta A_{N_2}^\top (\tilde y^*+\lambda (v_x)_{N_2})-\eta c_{N_2})\\
            & \qquad\qquad\qquad \ \ \  = \text{proj}_{\mathbb R^{|N_2|}_+}(-s A_{N_2}^\top (\tilde y^*+\lambda (v_x)_{N_2})-\eta c_{N_2})=0=\tilde x_{N_2}^*+(\lambda+1)(v_x)_{N_2} \ .\\
            &T(\tilde z^*+\lambda v)_y =\tilde y^*+\lambda v_y+\eta A(2(\tilde x^*+(\lambda+1)v_x)-(\tilde x^*+\lambda v_x))-\eta b\\
            & \qquad\qquad\ \ \ \  =\tilde  y^*+\eta A(\tilde x^*+(\lambda+2) v_x)-\eta b+\lambda v_y\\
            & \qquad\qquad\ \ \ \  =\tilde y^*+(\lambda+1) v_y \ .
        \end{aligned}
    \end{equation*}
    Therefore, we have $T(\tilde z^*+\lambda v)-(\tilde z^*+\lambda v)=v$ for any $\tilde z^*\in \mathrm{Fix}(\widetilde T)\cap\{z:x_{N_2}=0\}$.
\end{proof}

Next, we denote $\lambda_{k+1}:=\frac{k+1}{k+2}(\lambda_k+1)+\frac{1}{k+2}\lambda_0$ and $\hat z^{k+1}:=z^{k+1}-\lambda_{k+1}v$. The following lemma shows that $\hat z^k$ is bounded.
\begin{lem}\label{lem:infeas-nonexpansive}
    Consider $\{z^k\}$ the iterates of Halpern PDHG. Then for any $\tilde z^*_0\in \mathrm{Fix}(\widetilde T)\cap\{z:x_{N_2}=0\}$, it holds for any $k\geq 0$ that
    \begin{equation*}
        \|\hat z^{k}-\tilde z^*_0\|\leq \|\hat z^0-\tilde z^*_0\| \ ,
    \end{equation*}
    and thus
    \begin{equation*}
        \|\hat z^k\|\leq \|\hat z^0-\tilde z^*_0\|+\|\tilde z^*_0\|,\quad\ \|\hat z^k\|_2\leq 2(\|\hat z^0-\tilde z^*_0\|_2+\|\tilde z^*_0\|_2) \ .
    \end{equation*}
\end{lem}
\begin{proof}
We prove by induction. Suppose for some $k\geq 0$, $\|\hat z^{k}-\tilde z^*_0\|\leq \|\hat z^0-\tilde z^*_0\|$ and then
    \begin{equation*}
        \begin{aligned}
            \|\hat z^{k+1}-\tilde z^*_0\| &= \|z^{k+1}-\lambda_{k+1}v-\tilde z^*_0\|=\left\| \frac{k+1}{k+2}T(z^k)+\frac{1}{k+2}z^0-\tilde z^*_0-\lambda_{k+1}v\right\|\\
            & \leq \frac{k+1}{k+2}\|T(z^k)-\tilde z^*_0-(\lambda_k+1)v\|+\frac{1}{k+2}\|z^0-\tilde z^*_0-\lambda_0v\|\\
            & = \frac{k+1}{k+2}\|T(z^k)-T(\tilde z^*_0+\lambda_kv)\|+\frac{1}{k+2}\|\hat z^0-\tilde z^*_0\|\\
            & \leq \frac{k+1}{k+2}\|z^k-\tilde z^*_0-\lambda_kv\|+\frac{1}{k+2}\|\hat z^0-\tilde z^*_0\|=\frac{k+1}{k+2}\|\hat z^k-\tilde z^*_0\|+\frac{1}{k+2}\|\hat z^0-\tilde z^*_0\| \leq \|\hat z^0-\tilde z^*_0\| \ ,
        \end{aligned}
    \end{equation*}
    where the third inequality uses $T(\tilde z_0^*+\lambda_kv)=\tilde z_0^*+\lambda_kv+v$ since $\tilde z_0^*+\lambda_kv\in \mathcal Z_v^*$ and the second inequality utilizes the non-expansiveness of PDHG operator $T$. This implies for any $k$, $\|\hat z^{k}-\tilde z^*_0\|\leq \|\hat z^0-\tilde z^*_0\|$and thus for any $k\geq 0$.
    \begin{equation*}
        \|\hat z^k\|\leq \|\hat z^k-\tilde z^*_0\|+\|\tilde z^*_0\|\leq \|\hat z^0-\tilde z^*_0\|+\|\tilde z^*_0\| \ .
    \end{equation*}
    The last inequality utilizes Lemma \ref{lem:change-of-norm}.
\end{proof}

The following lemma implies that sharpness holds after identification.
\begin{lem}\label{lem:infeas-sharp}
    Consider $\{z^k\}$ the iterates of Halpern PDHG. Suppose there exists $K$ such that for any $k\geq K$, $(T(z^k)_x)_B>0$ and $(T(z^k)_x)_{N_2}=0$. Then there exists an $\alpha_\eta>0$ such that for any $k\geq K$
    \begin{equation*}
        \alpha_\eta\mathrm{dist}(T(z^k),\mathcal Z_v^*\leq \|T(z^k)-z^k-v\| \ .
    \end{equation*}
\end{lem}
\begin{proof}
By definition of $\alpha_\eta$, we know that for any $z=(x,y)$ satisfying $\|z\|_2\leq 2(\|\hat z^0-\tilde z^*_0\|_2+\|\tilde z^*_0\|_2)$ and $x_{N_2}=0$, it holds that
    \begin{equation*}
        \alpha_\eta\mathrm{dist}(z,\mathrm{Fix}(\widetilde T)\cap\{z:x_{N_2}=0\})\leq \|\widetilde T(z)-z\| \ .
    \end{equation*}
It follows from part (iii) of Lemma \ref{lem:auxiliary-lp} that $\{\tilde z^*+\lambda v:\tilde z^*\in \mathrm{Fix}(\widetilde T)\cap\{z:x_{N_2}=0\},\lambda \geq 0\} \subseteq \mathcal Z_v^*$ and thus for $\tilde z^*=\argmin_{u\in \mathrm{Fix}(\widetilde T)\cap\{z:x_{N_2}=0\}}\|u-\widetilde T(\hat z^k)\|$ we have
    \begin{equation*}
        \begin{aligned}
            \mathrm{dist}(T(z^k),\mathcal Z_v^*) & \leq \mathrm{dist}(T(z^k),\{\tilde z^*+\lambda v:\tilde z^*\in \mathrm{Fix}(\widetilde T)\cap\{z:x_{N_2}=0\},\lambda \geq 0\})\\
            & \leq \|T(z^k)-(\lambda_k+1)v-\tilde z^*\|=\|\widetilde T(\hat z^k)-\tilde z^*\|=\mathrm{dist}(\widetilde T(\hat z^k), \mathrm{Fix}(\widetilde T)\cap\{z:x_{N_2}=0\})\\
            & \leq \frac{1}{\alpha_\eta} \|\widetilde T(\hat z^k)-\widetilde T(\widetilde T(\hat z^k))\|\leq \frac{1}{\alpha_\eta} \|\hat z^k-\widetilde T(\hat z^k)\|\\
            & = \frac{1}{\alpha_\eta} \|z^k-\lambda_k v-\widetilde T(z^k-\lambda_k v)\|=\frac{1}{\alpha_\eta} \|T(z^k)-z^k-v\| \ ,
        \end{aligned}
    \end{equation*}
    where the third inequality uses $(\widetilde T(\hat z^k)_x)_{N_2}=0$ and $\|\widetilde T(\hat z^k)\|\leq \|\hat z^0-\tilde z^*_0\|+\|\tilde z^*_0\|$ from Lemma \ref{lem:infeas-nonexpansive}. The first and last equality is due to part (i) of Lemma \ref{lem:auxiliary-lp}.
\end{proof}

Now, we are ready to prove Theorem \ref{thm:infeas}:
\begin{proof}[Proof of Theorem \ref{thm:infeas}]
Following the notations in Lemma \ref{lem:infeas-nonexpansive}, we denote $\hat z^{n,k}=z^{n,k}-\lambda_{n,k}v$ with $\lambda_{n,k}$ defined as
\begin{equation*}
    \lambda_{n,k+1}=\frac{k+1}{k+2}(\lambda_{n,k}+1)+\frac{1}{k+2}\lambda_{n,0}, \ \lambda_{n,0}=\lambda_{n-1,k^*}+1, \ \lambda_{0,0}=0 \ .
\end{equation*}
By utilizing Lemma \ref{lem:infeas-nonexpansive} and suppose $\|\hat z^{n,0}-\tilde z_0^*\|\leq \|z^{0,0}-\tilde z_0^*\|$ where $\tilde z_0^* \in \mathrm{Fix}(\widetilde T)\cap\{z:x_{N_2}=0\}$, it turns out that
\begin{equation*}
\begin{aligned}
    \|\hat z^{n+1,0}-\tilde z_0^*\|& =\|T(z^{n,k^*})-(\lambda_{n,k^*}+1)v-\tilde z_0^*\|\leq \|z^{n,k^*}-\lambda_{n,k^*}v-\tilde z_0^*\|=\|\hat z^{n,k^*}-\tilde z_0^*\|\\
    &\leq \|\hat z^{n,0}-\tilde z_0^*\| \leq \|z^{0,0}-\tilde z_0^*\|
\end{aligned}
\end{equation*}
where the first equality uses $z^{n+1,0}=T(z^{n,k})$ and the first inequality utilizes non-expansiveness of PDHG operator $T$. The second inequality follows from Lemma \ref{lem:infeas-nonexpansive} and the last one follows from induction assumption. Hence $\|\hat z^{n,0}-\tilde z_0^*\|\leq \|z^{0,0}-\tilde z_0^*\|$ for any $n$ by induction, and it holds that
\begin{equation*}
    \begin{aligned}
        \mathrm{dist}(z^{n+1,0},\mathcal Z_v^*)& \leq \|z^{n+1,0}-(\lambda_{n,k^*}+1)v-\tilde z_0^*\|=\|\hat z^{n+1,0}-\tilde z_0^*\|\leq\|z^{0,0}-\tilde z_0^*\|
    \end{aligned}
\end{equation*}
where the first inequality is from part (iii) of Lemma \ref{lem:auxiliary-lp}.

Thus we know that with the choice of restart frequency $k\geq \frac{12R}{C\delta_v}\geq \frac{6\sqrt{2\eta}\|z^{0,0}-\tilde z_0^*\|}{C\delta_v}$, it always holds that $k\geq \frac{6\sqrt{2\eta}\|z^{0,0}-\tilde z_0^*\|}{C\delta_v}\geq \frac{6\sqrt{2\eta}\mathrm{dist}(z^{n,0},\mathcal Z_v^*)}{C\delta_v}$ for any $n\geq 0$. This implies $z^{n,k^*}$ satisfies $(T(z^{n,k^*})_x)_B>0,\ (T(z^{n,k^*})_x)_{N_2}=0$ for any $n\geq 0$ and $k^*\geq \frac{6\sqrt{2\eta}\|z^{0,0}-\tilde z_0^*\|}{C\delta_v}$.

Furthermore, it is straightforward to check that $\frac{\sqrt{\log k^*+5}+1}{k^*+1}\frac{1}{\alpha_\eta}\leq \frac{1}{e}$ from $k^*\geq \frac{(\sqrt 2+1)e}{\alpha_\eta}\log\pran{\frac{(\sqrt 2+1)e}{\alpha_\eta}}$. Now we are ready to prove the linear convergence of $\mathrm{dist}(z^{n,0},\mathcal Z_v^*)$ by combining Lemma \ref{lem:infeas-sharp} and Lemma \ref{lem:infeas-iden} -- it holds that
\begin{equation}\label{eq:infeas-dist-linear}
    \begin{aligned}
        \mathrm{dist}(z^{n,0},\mathcal Z_v^*)&=\mathrm{dist}(T(z^{n-1,k^*}),\mathcal Z_v^*)\leq \frac{1}{\alpha_\eta} \|T(z^{n-1,k^*})-z^{n-1,k^*}-v\|\\
        & \leq \frac{1}{\alpha_\eta}\frac{\sqrt{\log k^*+5}+1}{k^*+1}\mathrm{dist}(z^{n-1,0},\mathcal Z_v^*)\\
        & \leq e^{-1}\mathrm{dist}(z^{n-1,0},\mathcal Z_v^*)
        \leq \ldots \leq e^{-n}\mathrm{dist}(z^{0,0},\mathcal Z_v^*) \ ,
    \end{aligned}
\end{equation}
where the first and second inequalities are due to Lemma \ref{lem:infeas-sharp} and Lemma \ref{lem:sublinear-infeas}, respectively. Thus by noticing $\|T(z^{n,0})-z^{n,0}-v\|\leq \|T(z^{n,0})-z^{n,0}-T(z^*)+z^*\|\leq 2\|z^{n,0}-z^*\|$ for any $z^*$ such that $T(z^*)-z^*=v$, we have
\begin{equation*}
    \|T(z^{n,0})-z^{n,0}-v\|\leq 2\mathrm{dist}(z^{n,0},\mathcal Z_v^*) \leq 2e^{-n}\mathrm{dist}(z^{0,0},\mathcal Z_v^*) \ ,
\end{equation*}
which concludes the proof.
\end{proof}

\section{{Extension: reflected restarted Halpern PDHG ($\text{r}^2$HPDHG)}}

In previous sections, we propose and analyze restarted Halpern PDHG (rHPDHG) for solving LP, which is essentially Halpern iteration on PDHG operator with restart. In this section, we discuss an enhancement over rHPDHG that enables potentially faster convergence. The extension is straightforward: instead of using Halpern iteration upon vanilla PDHG operator $T$, we use Halpern iteration on its reflection $2T-I$~\cite{ryu2016primer,lieder2021convergence}. The intuition is that PDHG operator $T$ is a firmly non-expansive operator. On the other hand, the sublinear convergence guarantee of Halpern iteration just requires non-expansive operators, which is a weaker condition than firm non-expansiveness. Indeed, it is well known that for a firmly non-expansive operator $T$, its reflection $2T-I$ is a non-expansive operator \cite{bauschke2019convex}, thus we can apply Halpern on reflected PDHG. Effectively, the use of reflection takes a longer step and thus can improve the complexity result by a constant factor, which we discuss below. This extension with reflection is inspired by a recent work \cite{chen2024hpr}, where a similar technique is adopted for a Halpern Peacemen-Rachford method on solving LP with an improved numerical performance.

More formally, the iterate update of reflected Halpern PDHG is defined as
\begin{equation}\label{eq:hrpdhg}
    z^{k+1}=\text{reflected-H-PDHG}(z^k;z^0):=\frac{k+1}{k+2}(2T(z^k)-z^k)+\frac{1}{k+2}z^0 \ ,
\end{equation}
where $T$ is the PDHG operator defined in \eqref{eq:pdhg}. Algorithm \ref{alg:hpdhg-restart-reflected} formally presents the reflected restarted Halpern PDHG ($\text{r}^2$HPDHG)  algorithm. It is very similar to rHPDHG algorithm stated in Algorithm \ref{alg:hpdhg-restart}, yet with an additional reflection step~\eqref{eq:hrpdhg}. Similar to Algorithm \ref{alg:hpdhg-restart}, we can use either fixed frequency restart or adaptive restart \eqref{eq:adaptive-restart}. 

\begin{algorithm}
\caption{Reflected restarted Halpern PDHG ($\mathrm{r^2HPDHG}$) for \eqref{eq:minmax}}
\label{alg:hpdhg-restart-reflected}
\SetKwInOut{Input}{Input}
\Input{Initial point $z^{0,0}$, outer loop counter $n\leftarrow 0$.}

\Repeat{\upshape $z^{n+1,0}$ convergence}{
  initialize the inner loop counter $k\leftarrow0$;\\
  \Repeat{\upshape restart condition holds}{
    $z^{n,k+1}\leftarrow \text{reflected-H-PDHG}(z^{n,k};z^{n,0})$;
  }
  initialize the initial solution $z^{n+1,0}\leftarrow T(z^{n,k})$;\\ 
  $n\leftarrow n+1$;
}
\end{algorithm}

We next explain how the reflection can lead to a factor of $2$ improvement in the theoretical complexity compared to rHPDHG. We start with the following proposition, which essentially says the reflected PDHG is a non-expansive operator:

\begin{prop} \cite[Proposition 4.4]{bauschke2019convex}
    Let $T$ be PDHG operator for solving \eqref{eq:minmax}. Then its reflected operator $2T-I$ is non-expansive .
\end{prop}

We next illustrate how to obtain a factor-of-$2$ theoretical improvement of $\mathrm{r^2HPDHG}$ over rHPDHG.

First, notice that the sub-linear convergence of Halpern reflected PDHG follows directly from the non-expansiveness of the reflected operator.
\begin{lem}[{\cite[Theorem 2.1 and its proof (page 4)]{lieder2021convergence}}]\label{lem:sublinear-feas-reflected}
    Consider $\{z^{k}\}$ the iterates of Halpern reflected PDHG on solving feasible \eqref{eq:lp}. Denote $z^*\in \mathcal Z^*:=\{z^*\mid \PDHG(z^*)-z^*=0\}$. Then it holds for any $k\geq 1$ that
    \begin{equation*}
        \left\|(2\PDHG-I)(z^k)-z^k\right\|\leq \frac{2}{k+1}\mathrm{dist}(z^0,\mathcal Z^*), \ \mathrm{and} \ \left\|(2\PDHG-I)(z^k)-z^k\right\|\leq \frac{2}{k}\left\|z^k-z^0\right\|
    \end{equation*}
\end{lem}
An immediate consequence of Lemma \ref{lem:sublinear-feas-reflected} is that reflection improves the convergence rate of the fixed-point residual by a factor of two, compared to the regular Halpern PDHG stated in Lemma \ref{lem:sublinear-feas}:
\begin{cor}\label{cor:sublinear-feas-reflected}
Under the same condition as Lemma \ref{lem:sublinear-feas-reflected}, it holds for any $k\geq 0$ that
    \begin{equation*}
        \left\|\PDHG(z^k)-z^k\right\|\leq \frac{1}{k+1}\mathrm{dist}(z^0,\mathcal Z^*), \ \mathrm{and} \ \left\|\PDHG(z^k)-z^k\right\|\leq \frac{1}{k}\left\|z^k-z^0\right\| \ .
    \end{equation*}
\end{cor}
Compared with the Halpern PDHG with reflection (Lemma \ref{lem:sublinear-feas}), Corollary \ref{cor:sublinear-feas-reflected} shows that reflection improves the sublinear convergence rate by a factor of $2$. This factor of $2$ can be passed over to all of the theoretical guarantees stated in the previous section. For example, consider the situation of fixed frequency restart on feasible LP (i.e., Theorem \ref{thm:fixed-restart}) allow us to set the length of the inner loop as 
\begin{equation}\label{eq:fixed-restart-reflect}
    k^*= \left\lceil\frac{e}{\alpha_\eta}\right\rceil \ ,
\end{equation}
which is a factor of $2$ smaller than that in Theorem \ref{thm:fixed-restart}. With an identical proof of Theorem \ref{thm:fixed-restart} and utilizing Corollary \ref{cor:sublinear-feas-reflected} instead of Lemma \ref{lem:sublinear-feas}, we can obtain:
\begin{thm}[Fixed frequency restart with reflection for feasible LP]\label{thm:fixed-restart-reflect}
For a feasible and bounded LP~\eqref{eq:lp}, consider $\{z^{n,k}\}$ the iterates of reflected restarted Halpern PDHG (Algorithm \ref{alg:hpdhg-restart-reflected}) with fixed frequency restart scheme, namely, we restart the outer loop if \eqref{eq:fixed-restart-reflect} holds. Denote $\mathcal Z^*$ is the set of optimal solutions. Then it holds for any $n\geq 0$ that
    \begin{equation*}
        \mathrm{dist}(z^{n+1,0},\mathcal Z^*)\leq \pran{\frac 1e}^{n+1} \mathrm{dist}(z^{0,0},\mathcal Z^*) \ .
    \end{equation*}
\end{thm}

With the same logic, all other theoretical results stated in previous sections (i.e., Theorem \ref{thm:adaptive-restart}, Theorem \ref{thm:stage-1}, Theorem \ref{thm:stage-2} and Theorem \ref{thm:infeas}) can be derived for $\mathrm{r^2HPDHG}$ in parallel to the analysis for rHPDHG with a factor of $2$ improvement.

\section{{Numerical experiments}}
In this section, we present HPDLP, a GPU-based FOM LP solver based on rHPDHG/$\mathrm{r^2HPDHG}$. We compare rHPDHG and $\mathrm{r^2HPDHG}$ with cuPDLP, a GPU-implemented LP solver based on raPDHG, which was shown to have comparable performance to Gurobi for solving LP on standard benchmark sets~\cite{lu2023cupdlp}. The major message of this section is that rHPDHG can achieve comparable to cuPDLP while there is speedup of $\mathrm{r^2HPDHG}$ over cuPDLP, demonstrating their potential to be alternative base routines of FOM-based LP solver.

{\bf Benchmark dataset.} We use \texttt{MIP Relaxations} as our benchmark set, which contain 383 instances curated from root-node LP relaxation of mixed-integer programming problems from MIPLIB 2017 collection~\cite{gleixner2021miplib}. 383 instances are selected from MIPLIB 2017 to construct \texttt{MIP Relaxations} based on the same criteria as~\cite{applegate2021practical, lu2023cupdlp}. The dataset is further split into three classes based on the number of nonzeros (nnz) in the constraint matrix, as shown in Table \ref{tab:miplib-size}.
\begin{table}[ht!]
\centering
\begin{tabular}{cccc}
\hline
                    & \textbf{Small}                   & \textbf{Medium}                     & \textbf{Large}                   \\ \hline
\textbf{Number of nonzeros}  & 100K -  1M & 1M - 10M & \textgreater 10M \\
\textbf{Number of instances} & 269                     & 94                         & 20                      \\ \hline
\end{tabular}
\caption{Scales of instances in \texttt{MIP Relaxations}.}
\label{tab:miplib-size}
\end{table}

{\bf Software and computing environment.} HPDLP and cuPDLP.jl are both implemented in an open-source Julia~\cite{bezanson2017julia} module and utilize \href{https://github.com/JuliaGPU/CUDA.jl}{CUDA.jl}~\cite{besard2018effective} as the interface for working with NVIDIA CUDA GPUs using Julia. We use NVIDIA H100-PCIe-80GB GPU, with CUDA 12.3, for running experiments. The experiments are performed in Julia 1.9.2.

{\bf Termination criteria.} The same termination criteria as cuPDLP.jl is utilized~\cite{lu2023cupdlp}, namely, the solver will terminate when the relative KKT error is no greater than the termination tolerance $\epsilon\in(0,\infty)$.

{\bf Time limit.} We impose a time limit of 3600 seconds on instances with small-sized and medium-sized instances and a time limit of 18000 seconds for large instances.

{\bf Shifted geometric mean.} We report the shifted geometric mean of solve time to measure the performance of solvers on a certain collection of problems. More precisely, shifted geometric mean is defined as $\left(\prod_{i=1}^n (t_i+\Delta)\right)^{1/n}-\Delta$ where $t_i$ is the solve time for the $i$-th instance. We shift by $\Delta=10$ and denote it SGM10. If the instance is unsolved, the solve time is always set to the corresponding time limit. 

{\bf Algorithmic enhancements.} To improve the practical performance, cuPDLP exploits several algorithmic enhancements upon the base algorithm raPDHG, including preprocessing, adaptive step-size, primal weight, and adaptive restart based on KKT error. More details refer to~\cite{lu2023cupdlp}. rHPDHG/$\mathrm{r^2HPDHG}$ adopts the same heuristics as cuPDLP except the adaptive restart and leverage the fixed point residual as its restart metric (see \eqref{eq:adaptive-restart}).

{\bf Results.} Table \ref{tab:miplib-1e-4} and \ref{tab:miplib-1e-8} present the comparison between rHPDHG, $\mathrm{r^2HPDHG}$ and cuPDLP.jl. Table \ref{tab:miplib-1e-4} presents moderate accuracy (i.e., $10^{-4}$ relative KKT error) while Table \ref{tab:miplib-1e-8} presents high accuracy (i.e., $10^{-8}$ relative KKT error).
\begin{table}[h!]
\centering
{
\begin{tabular}{ccccccccc}
\hline
\multirow{2}{*}{}                                   & \multicolumn{2}{c}{\begin{tabular}[c]{@{}c@{}}\textbf{Small (269)} \\ (1-hour limit)\end{tabular}} & \multicolumn{2}{c}{\begin{tabular}[c]{@{}c@{}}\textbf{Medium (94)}\\ (1-hour limit)\end{tabular}}    & \multicolumn{2}{c}{\begin{tabular}[c]{@{}c@{}}\textbf{\textbf{Large (20)}}\\ (5-hour limit)\end{tabular}}  & \multicolumn{2}{c}{\textbf{Total (383)}}       \\
                                                    & \textbf{Count} & \textbf{Time} & \textbf{Count} & \textbf{Time} & \textbf{Count} & \textbf{Time} & \textbf{Count} & \textbf{Time} \\ \hline
\multicolumn{1}{c}{\textbf{rHPDHG}} & 267                   & 8.33                & 92                   & 12.70              & 17                    & 108.42   &376 &11.30           \\   
\multicolumn{1}{c}{{$\mathrm{\bf r^2HPDHG}$}}     & 267                   & 6.61               & 93                    & 7.84              & 19                    & 90.81 &379 &8.58         \\
\multicolumn{1}{c}{\textbf{cuPDLP.jl}}     & 266                   & 8.46               & 92                    & 13.76               & 19                    & 86.39 &377 &11.41            \\\hline
\end{tabular}
}
\caption{Solve time in seconds and SGM10 of different solvers on instances of \texttt{MIP Relaxations} with tolerance $10^{-4}$: rHPDHG/$\mathrm{r^2HPDHG}$ versus cuPDLP.jl.}
\label{tab:miplib-1e-4}
\end{table}

\begin{table}[h!]
\centering
{
\begin{tabular}{ccccccccc}
\hline
\multirow{2}{*}{}                                   & \multicolumn{2}{c}{\begin{tabular}[c]{@{}c@{}}\textbf{Small (269)} \\ (1-hour limit)\end{tabular}} & \multicolumn{2}{c}{\begin{tabular}[c]{@{}c@{}}\textbf{Medium (94)}\\ (1-hour limit)\end{tabular}}    & \multicolumn{2}{c}{\begin{tabular}[c]{@{}c@{}}\textbf{\textbf{Large (20)}}\\ (5-hour limit)\end{tabular}}  & \multicolumn{2}{c}{\textbf{Total (383)}}       \\
                                                    & \textbf{Count} & \textbf{Time} & \textbf{Count} & \textbf{Time} & \textbf{Count} & \textbf{Time} & \textbf{Count} & \textbf{Time} \\ \hline
\multicolumn{1}{c}{\textbf{rHPDHG}} & 260                   & 26.64               & 88                   & 38.89              & 16                    & 295.92   &364 &33.83           \\     
\multicolumn{1}{c}{{$\mathrm{\bf r^2HPDHG}$}}     & 260                   & 19.13               & 87                    & 28.35              & 16                    & 229.47 &363 &24.79       \\
\multicolumn{1}{c}{\textbf{cuPDLP.jl}}     & 261                  & 23.26               & 86                    & 39.20               & 16                    & 350.84 &363 &31.47            \\\hline
\end{tabular}
}
\caption{Solve time in seconds and SGM10 of different solvers on instances of \texttt{MIP Relaxations} with tolerance $10^{-8}$: rHPDHG/$\mathrm{r^2HPDHG}$ versus cuPDLP.jl.}
\label{tab:miplib-1e-8}
\end{table}

A notable observation is that rHPDHG achieves a performance comparable to that of cuPDLP.jl while $\mathrm{r^2HPDHG}$ even gains remarkable speedup over cuPDLP.jl, regardless of the problem size and solution accuracy. In particular, rHPDHG solves 376 out of 383 instances with average 11.3 seconds under moderate accuracy, while cuPDLP.jl solves 1 more instance with approximately the same solve time. Furthermore, $\mathrm{r^2HPDHG}$ solves 2 more instances than cuPDLP.jl with 1.33x speedup. In terms of high accuracy, rHPDLP solves 1 more instance than cuPDLP.jl while $\mathrm{r^2HPDHG}$ achieves 1.27x speedup over cuPDLP.jl. To summarize, rHPDHG has similar overall performances as cuPDLP.jl while $\mathrm{r^2HPDHG}$ exhibits a stronger overall performance.

\section{Conclusions}
In this paper, we present rHPDHG/$\mathrm{r^2HPDHG}$, a matrix-free primal-dual algorithm for solving LP. We show that the algorithm achieves accelerated rates to identify the active variables and eventual linear convergence on feasible LP, while it can recover infeasibility certificates with an accelerated linear rate. We further build up an LP solver based on rHPDHG/$\mathrm{r^2HPDHG}$ and GPUs, whose numerical experiments showcase the strong empirical performance for LP.

\bibliographystyle{amsplain}
\bibliography{ref-papers}

\end{document}